\documentclass[11pt]{article}

\newcommand{\C}{\mathcal{C}} %
\newcommand{\R}{\mathcal{R}} %
\newcommand{\Q}{\mathbb{Q}}
\newcommand{\Z}{\mathbb{Z}}
\usepackage[authoryear]{natbib}
\usepackage{fullpage}
\usepackage{authblk}
\usepackage{tikz}

\tikzset{Red/.style={red}}
\tikzset{Blue/.style={blue}}
\tikzset{ForestGreen/.style={green}}
\tikzset{Sepia/.style={brown}}
\tikzset{Black/.style={black}}
\usepackage[title]{appendix}
\renewcommand{\Re}{\mathbb{R}}
\newcommand{\vinf}{\operatorname{vinf}}

\renewcommand{\S}{\mathcal{S}}
\newcommand{\E}{\mathcal{S}^E} 
\newcommand{\F}{\mathcal{F}} 
\newcommand{\VF}{RVF}
\newcommand{\LPVF}{{RLPVF}}
\newcommand{\zLP}{z_{\rm LP}}
\newcommand{\CLP}{\C_{\rm LP}}
\newcommand{\hzeta}{\hat{\zeta}}
\newcommand{\PD}{{\mathcal P}_{D}} 
\newcommand{\Extr}{\mathcal{E}} 
\newcommand{\proj}{\operatorname{proj}}
\newcommand{\XMO}{X_{\rm MO}}
\newcommand{\Smin}{\S_{\min}}
\newcommand{\hu}{\hat{u}}
\newcommand{\inter}{\operatorname{int}}
\newcommand{\epi}{\operatorname{epi}}

\usepackage{amsmath,amsfonts,amsthm,commands,xcolor,graphicx,float,amssymb}
\usetikzlibrary{shapes,arrows,math}
\usepackage{pgfplots}
\pgfplotsset{compat=1.16}
\DeclareMathOperator*{\argmax}{arg\,max}
\DeclareMathOperator*{\argmin}{arg\,min}
\usepackage{enumerate} 
\usepackage{subcaption}
\usepackage{dirtytalk}
\usepackage[shortlabels]{enumitem}
\usepackage[nocomma]{optidef}

\usepackage{hyperref}
\hypersetup{
	colorlinks,
	citecolor=blue,
	filecolor=blue,
	linkcolor=blue,
	urlcolor=black
}
\allowdisplaybreaks
\usepackage[ruled,vlined,linesnumbered, titlenotnumbered,algo2e]{algorithm2e}
\usepackage{algorithm}
\DeclareCaptionLabelFormat{noname}{#2}
\usepackage{euscript}
\usepackage{etoolbox}
\makeatletter
\newenvironment{mysubequations}[1]
{%
	\addtocounter{equation}{-1}%
	\begin{subequations}
		\renewcommand{\theparentequation}{#1-}%
		\def\@currentlabel{#1}%
	}
	{%
	\end{subequations}\ignorespacesafterend
}
\makeatother

\date{July 8, 2024}

\newcommand{\ImageDir}{figures}

\newcommand{\ABSTRACT}{In this study, we investigate the connection between the efficient frontier (EF) of a general multiobjective mixed integer linear optimization problem (MILP) and the so-called \emph{restricted value function} (RVF) of a closely related single-objective MILP. In the first part of the paper, we detail the mathematical structure of the RVF, including characterizing the set of points at which it is differentiable, the gradients at such points, and the subdifferential at all nondifferentiable points. We then show that the EF of the multiobjective MILP is comprised of points on the boundary of the epigraph of the RVF and that any description of the EF suffices to describe the RVF and vice versa. Because of the close relationship of the RVF to the EF, we observe that methods for constructing the so-called value function (VF) of an MILP and methods for constructing the EF of a multiobjective optimization problem are effectively interchangeable. Exploiting this observation, we propose a generalized cutting-plane algorithm for constructing the EF of a multiobjective MILP that arises from an existing algorithm for constructing the classical MILP VF. The algorithm identifies the set of all integer parts of solutions on the EF. We prove that the algorithm converges finitely under a standard boundedness assumption and comes with a performance guarantee if terminated early.}

\graphicspath{{\ImageDir/}}

\begin{document}

	\title{On the Relationship Between the Value Function and the Efficient Frontier of a Mixed Integer Linear Optimization Problem}
	\author[1]{Samira Fallah\thanks{\texttt{saf418@lehigh.edu} 
	}}
	\author[1]{Ted K. Ralphs\thanks{\texttt{ted@lehigh.edu}
	}}
	\author[2]{Natashia L. Boland\thanks{\texttt{natashia.boland@gmail.com}
	}}
	\affil[1]{Department of Industrial and Systems Engineering, Lehigh University, USA}
	\affil[2]{School of Electrical Engineering, Computing and Mathematical Sciences, Curtin University, Perth, Western Australia}

	\maketitle
	
	\begin{abstract}
		\ABSTRACT
	\end{abstract}
	
	
	\section{Introduction}
	\label{sec:intro}
	In this study, we consider the relationship between the efficient frontier (EF) of a multiobjective mixed integer linear optimization problem (MILP) and a certain value function (VF), which we refer to as the \emph{restricted value function} (\VF{}), associated with a closely related mixed integer linear optimization problem. Our main result is that any description of the EF suffices to describe the RVF and vice versa. Specifically, we demonstrate that the EF is composed of a subset of the points on the boundary of the epigraph of the \VF{} (also referred to as the \emph{graph} of the \VF{}), and that any point on the boundary and \emph{not} in the EF has (i) the same objective value as a point on the EF that dominates it, and (ii) a directional derivative of zero, in the direction towards the dominating point. 
 
	Although our results demonstrate that the EF and the RVF are closely related and indeed effectively interchangeable, their relationship seems not to have been previously observed mainly because of the disparate ways in which these mathematical objects have been described in the separate literatures in which the concepts have been developed. Upon closer examination, however, the relationship between the \VF{} and the EF is quite intuitive. While we believe this work is the first to formally and explicitly establish the relationship, it can be seen implicitly in the results of several previous works, such as those by~\citep{trapp2013level},~\citep{ralphs2014value}, and~\citep{bodur2016decomposition}. The so-called minimal tenders utilized in~\citeauthor{trapp2013level}'s algorithm for constructing the VF of a pure integer linear optimization problem (PILP) can be seen as the points on the EF of a related multiobjective problem. \citet{ralphs2014value} generalized this concept in their work on the structure of the VF of a general MILP. More recently,~\citet{bodur2016decomposition} observed that in block-structured problems, the solution to the column generation subproblem can be viewed as equivalent to evaluating a certain VF, and solutions can thus be restricted only to so-called \emph{nondominated points}.
	
	The relationship described in the remainder of the paper has some apparently broad-ranging implications, including that algorithms designed for the construction and/or approximation of the EF are effectively interchangeable with algorithms for the construction and/or approximation of the \VF{} and VFs in general. Because algorithms for these two tasks have so far been used in very different application domains and for very different purposes, there are likely many possibilities for the cross-pollination of ideas. To illustrate this, we propose a generalized cutting-plane algorithm for constructing both the \VF{} and the EF. The approach we suggest is finitely convergent, exploits the discrete structure of the \VF{}, and provides a performance guarantee if terminated early. It is a modified version of an existing algorithm for constructing the full VF, and to the best of our knowledge, the approach is entirely different from existing algorithms for the construction of the EF. Additionally, our algorithm is one of few algorithms developed to date that addresses multiobjective MILPs in the presence of continuous variables with any number of objectives, and it yields improved bounds on the number and size of subproblems that need to be solved to determine the discrete structure of the EF. These bounds are also comparable to existing algorithms for the PILP case (without continuous variables).

	In the remainder of this section, we set the stage by formally defining the important terms and concepts. We first describe the terminology and basic properties related to multiobjective MILPs and their associated EFs before introducing the concept of the \VF{}. Although the \VF{} that we introduce is closely related to the classical VF of a single-objective MILP (it can be viewed as a generalization), we are not aware of any previous study of it. Its properties are much more difficult to characterize than those of the classical VF. 
	
	\paragraph{Multiobjective Optimization.} The multiobjective MILP that serves as the focus of our study is defined as 
	\begin{equation}
		\vinf_{(x_I, x_C) \in \XMO} C_I x_I + C_C x_C,
		\tag{MO-MILP}\label{eq.MultiObj_ch2}
	\end{equation}
	where 
	\begin{equation*}
		\XMO = \left\{(x_I,x_C)\in \Z_+^r\times \Re_+^{n-r}\;: \; A_I x_I + A_C x_C = b\right\},
	\end{equation*}
	is the feasible region; $A \in \Q^{m \times n}$ is the coefficient matrix of the constraints; $b\in \Q^{m}$ is the right-hand side (RHS) of the constraints; and the rows of matrix $C \in \Q^{(\ell + 1) \times n}$ are the multiple objectives of the problem and denoted by $\{c^0, c^1, \hdots , c^{\ell}\}$. The $\vinf$ operator indicates that this is a vector minimization (multiobjective) problem, which means that there is not a single optimal value; the $\vinf$ operator returns a set of nondominated vectors of objective values, as described below \footnote{We use $\inf$ rather than $\min$ because the definition of the RVF below introduces the possibility of MILPs with irrational right-hand sides for which attainment of the optimal value is not guaranteed}. $A_I$ and $C_I$ are the submatrices of $A$ and $C$ consisting of columns associated with the integer variables (indexed by set $I = \{0, \hdots, r - 1\}$), as opposed to $A_C$ and $C_C$, which are the submatrices corresponding to the columns associated with the continuous variables (indexed by set $C = \{r, \hdots, n - 1\}$). We assume that the feasible region $\XMO$ is bounded.
	
	The aim of multiobjective optimization is to characterize the trade-offs inherent in optimizing multiple objectives simultaneously. This analysis is most naturally done in the $(\ell + 1)$-dimensional space known as the \emph{criterion space}, which contains the vectors of objective values associated with points in the $n$-dimensional \emph{decision space}, the space containing the feasible region $\XMO$. While solving an MILP with a single objective means determining its unique optimal value, solving a multiobjective MILP means generating the set of all vectors in criterion space associated with the so-called efficient solutions, those for which there is no other solution for which the objective value is at least as good for every objective and strictly better for at least one objective. 
	
	We briefly review some concepts in multiobjective optimization, referring interested readers to \citep{ehrgott2005multicriteria} for more details. An important concept in this context is that of dominance. The point $C_I x_I + C_C x_C \in \Re^{\ell+1}$ in criterion space, associated with $(x_I, x_C) \in \XMO$, \emph{dominates} $C_I x^\prime_I + C_C x^\prime_C \in \Re^{\ell+1}$, associated with $(x^\prime_I, x^\prime_C) \in \XMO$, if $C_I x_I + C_C x_C \lneqq C_I x^\prime_I + C_C x^\prime_C$, i.e., $(C_I x_I + C_C x_C)_j \leq (C_I x^\prime_I + C_C x^\prime_C)_j$ for all $j = \{0, 1, \hdots, l\}$ and $(C_I x_I + C_C x_C)_j < (C_I x^\prime_I + C_C x^\prime_C)_j$ for at least one index $j \in \{0, 1, \hdots, l\}$. A point in criterion space that is not dominated by any other point is called a \emph{nondominated point} (NDP). The 
 set of all NDPs is the aforementioned \emph{efficient frontier}. A preimage of an NDP in the decision space is referred to as an \emph{efficient solution}. A point $(x_I, x_C) \in \XMO$ that is not necessarily efficient but for which there does not exist $(x_I^\prime, x_C^\prime) \in \XMO$ such that $C_I x_I^\prime + C_C x_C^\prime < C_I x_I + C_C x_C$ is called \emph{weakly efficient}, and the associated point $C_I x_I + C_C x_C$ in criterion space is referred to as a \emph{weakly nondominated point} or a \emph{weak NDP}. It is important to note that a weak NDP that is not also an NDP is, in fact, a \textit{dominated} point. 

	\paragraph{Restricted Value Function.} What we call the \VF{} provides another way of analyzing the trade-offs in the multiobjective optimization problem. Specifically, we consider the following related MILP with a single objective obtained by imposing all but one of the objectives in~\eqref{eq.MultiObj_ch2} as constraints, which we refer to as the \emph{parametric constraints}. This MILP can be written as follows:
	\begin{equation}
		\inf_{(x_I, x_C)\in X} c^{0}_I x_I + c^{0}_{C} x_C, \tag{MILP} \label{eq.MILP_ch2}
	\end{equation}
	where 
	\begin{equation*}
		X = \left\{(x_I, x_C)\in \Z_+^r\times \Re_+^{n-r}\;: \; C_I^{1:\ell} x_I + C_C^{1:\ell} x_C \leq f,\ A_I x_I + A_C x_C = b\right\},
	\end{equation*}
	is the feasible region; $c^0$ is the first row of the matrix $C$; $C^{1:\ell}$ is the submatrix consisting of the remaining rows of $C$, and $f \in \Q^{\ell}$ is a fixed vector to be replaced shortly by a parameter to obtain the aforementioned \VF{}. The particular objective that is chosen as $c^0$ is arbitrary and the results hold no matter what objective is chosen.
	
	We now define the \VF{} $z: \Re^{\ell} \rightarrow \Re \cup \{\pm \infty\}$ associated with~\eqref{eq.MILP_ch2} to be the function 
	\begin{equation}
		z(\zeta) = \inf_{(x_I,x_C)\in \mathcal{S}(\zeta)} c^{0}_I x_I + c^{0}_C x_C, \tag{RVF} \label{eq.RVF_ch2}
	\end{equation}
	that returns the optimal solution value of~\eqref{eq.MILP_ch2} as a function of a RHS parameter $\zeta \in \Re^{\ell}$, where
	\begin{equation*}
		\mathcal{S} (\zeta) = \left\{(x_I, x_C)\in \Z_+^r \times \Re_+^{n-r}\;: \; C_I^{1:\ell} x_I + C_C^{1:\ell} x_C \leq \zeta,\ A_I x_I + A_C x_C = b\right\}.
	\end{equation*}
	Note that the classical VF~\citep{blair1977value,blair1979value,guzelsoy2007duality,ralphs2014value} is closely related and is the special case of the RVF in which the only constraints with a fixed right-hand side are the bound constraints ($m = 0$). %
 As usual, we let $z(\zeta) = +\infty$ for $\zeta \not\in \C$, where
	\begin{equation*}
		\C = \left\{\zeta \in \Re^{\ell}\;: \; \mathcal{S}(\zeta)\neq \emptyset\right\}.
	\end{equation*}
	The function $z$ is always bounded from below because of our assumption that the feasible region $\XMO$ of the multiobjective MILP is bounded.
	
	\bexample\label{ex.MILPVF}
	Here, we illustrate the concepts discussed so far. Consider an RVF defined by 
	\begin{equation*}
		\begin{array}{ll@{}ll}
			z(\zeta) = &\min~ &2x_1 + 5x_2 + 7x_4 + 10x_5 + 2x_6 + 10x_7\\
		& \emph{s.t.}	&-x_1 - 10x_2 + 10x_3 - 8x_4 + x_5 - 7x_6 + 6x_7 \leq \zeta\\
		&	&-x_1 + 4x_2 + 9x_3 + 3x_4 + 2x_5 + 6x_6 - 10x_7 = 4\\
		&	&  5x_2 + x_4 + x_8 = 5\\
		&	& 5x_2 + x_7 + x_9 = 5\\
		&	&x_{j} \in \{0, 1\}, \quad \forall j \in \{1, 2\}\\
		&	&x_{j} \in \Re_+, \quad \forall j \in \{3, 4, \hdots, 9\},
		\end{array}
		\qquad
	\end{equation*}
	for all $\zeta \in \Re$. This is not precisely in our standard form because the fixed upper bounds on the integer variables are embedded in their definition. However, it can easily be converted to our standard form.
	Figure~\ref{Fig:vf_first_instance} below shows the VF for the MILP, while Figure~\ref{Fig:EF_first_instance} shows the EF for the associated multiobjective optimization problem. Note that the graph of the VF{} and the EF are identical except for the horizontal line segment between the points $(-10,5)$ and $(-7\frac{5}{6},5)$, which has been thickened in the figure for emphasis.
	\begin{figure}[htbp]
		\begin{subfigure}{0.5\textwidth}\centering
			\includegraphics[width=\linewidth]{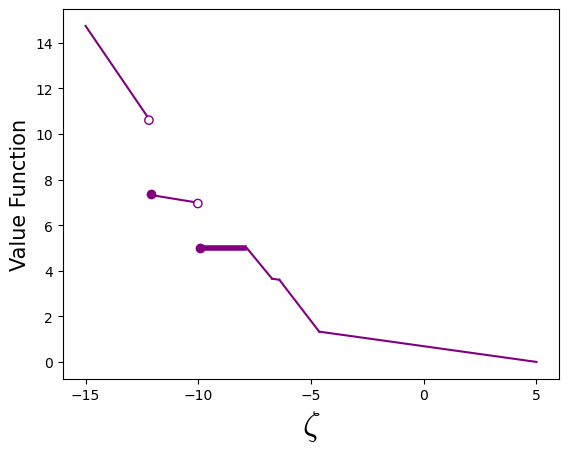}
			\caption{\VF{}}\label{Fig:vf_first_instance}
		\end{subfigure}\hfil
		\begin{subfigure}{0.5\textwidth}\centering
			\includegraphics[width=\linewidth]{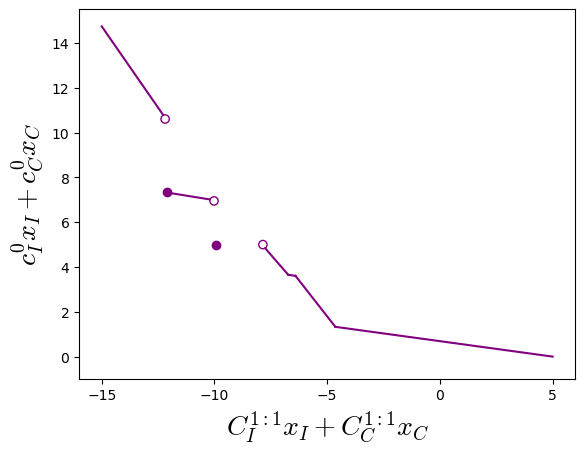}
			\caption{EF}\label{Fig:EF_first_instance}
		\end{subfigure}
		\caption{The portion of the \VF{} and the EF associated with Example~\ref{ex.MILPVF} with $\zeta\geq -15$. Note that the frontier extends to the left up to $\zeta=-57\frac{2}{3}$; the finite domain of $z$ is $\C=[-57\frac{2}{3},+\infty)$. The complete EF is shown in Figure~\ref{Fig:completeFrontier}.}
		\label{Fig:first_Ex}
	\end{figure}
	\eexample
	
	In the remainder of the study, we demonstrate that the RVF and the EF capture the same information and that algorithms for the construction of the two are effectively interchangeable. Throughout the study, we consider the given instance~\eqref{eq.MILP_ch2}, with its associated \VF{} and the corresponding instance of~\eqref{eq.MultiObj_ch2}. The study is organized as follows. We begin by reviewing related work in Section~\ref{sec:RelatedWork}. In Section~\ref{sec:MILP_VF}, we provide a characterization of the \VF{} in terms of a discrete set of integer parts of NDPs. In Section~\ref{sec:Relationship}, we formalize the relationship between the \VF{} and the EF. In Section~\ref{sec:FiniteAlgorithm}, we present our cutting-plane algorithm for constructing both the EF and the VF. Both the VF representation and the cutting-plane algorithm are finite under our assumption that $\XMO$ is bounded. We analyze the theoretical performance of the algorithm compared with that of existing algorithms. Finally, we summarize our findings and concluding remarks and suggest directions for future work in Section~\ref{sec:conclusions}.
	
	
	\section{Related Work}
	\label{sec:RelatedWork}
	Methods both for constructing the EF of a multiobjective optimization problem and for constructing the VF of an MILP have been extensively studied in the open literature. As the literature is vast, we focus here on the most closely related works. 
 
	\subsection{Multiobjective Optimization}
	Multiobjective optimization, the analysis of trade-offs between multiple conflicting objective functions, has numerous applications across various fields, as most real-world problems arising in practice do have multiple objectives. In this study, we address the generation of the exact EF for multiobjective MILPs. For a recent survey and a comprehensive overview of algorithms, we recommend~\citep{halffmann2022exact}. 
 
 Algorithms in the multiobjective area can be roughly classified into two categories: scalarization-based and non-scalarization-based methods. 
 Scalarization techniques are the most common and involve solving a sequence of single-objective problems, each one producing a single NDP or otherwise yielding information about a local region of the EF.
 Non-scalarization methods, on the other hand, take a more global view, using, e.g., an outer approximation to bound the EF by convex or nonconvex bounding functions. In the MILP case, some variant of branch-and-bound is typically used to implement these algorithms.
	
A wide variety of scalarization-based methods have been proposed, including the \emph{weighted sum} method~\citep{zadeh1963optimality}, the \emph{perpendicular search} method~\citep{chalmet1986algorithm}, the \emph{weighted Tchebycheff} method~\citep{bowman1976relationship, yu1973class, zeleny1973compromise}, the \emph{$\epsilon$-constraint} method \citep{haimes1971bicriterion}, the Hybrid method~\citep{guddat1985multiobjective}, Benson's method~\citep{benson1978existence}, and the Pacoletti-Serafini's method~\citep{pascoletti1984scalarizing}. These methods typically involve an iterative process in which a list of discovered NDPs and a list of unexplored regions of the criterion space are maintained. At each iteration, the algorithm produces a new NDP within an unexplored region of the criterion space. The process is repeated until there are no more unexplored regions. For a comprehensive review of scalarization methods in the multiobjective optimization field, the reader is referred to~\citep{ehrgott2006discussion}.
	
	The most straightforward scalarization method is the weighted sum method. The single objective created by this method is a weighted sum of the original objectives. When the weights are all positive, the solution to the weighted sum problem is guaranteed to be nondominated. On the other hand, not all nondominated solutions can be generated as a solution to some weighted sum problem in the MILP case---only the so-called \emph{supported NDPs} can be generated in this way. The NDPs that can be found via weighted sum-based scalarization are called \emph{supported} NDPs. 
 
	Several methods have been developed to address unsupported NDPs.
 The augmented weighted Tchebycheff method~\citep{bowman1976relationship} seeks to find NDPs within the exploration region by minimizing the distance to the ideal point, which is defined as the point whose components are obtained by minimizing the objective functions. \citet{ralphs2006improved} proposed a weighted Tchebycheff scalarization algorithm for constructing the EF of a biobjective integer programming problem.  
 
 The $\epsilon$-constraint method, which involves minimizing a primary objective while restricting the other objectives through inequality constraints, is widely used with many existing variants. This algorithmic approach has an obvious connection to the RVF. \citet{tamby2021enumeration} has recently expanded the use of the $\epsilon$-constraint method to address problems involving two or more objectives. The proposed algorithm divides the search area into segments that can be individually explored by solving an integer program.
    	
	To the best of our knowledge, the GoNDEF algorithm developed by~\citep{rasmi2019gondef} is currently the only algorithm that utilizes scalarization techniques to address MILPs with more than two objectives. GoNDEF first finds all integer parts of NDPs (as our proposed algorithm also does) and then attempts to describe the non-dominated facets of the stability regions associated with each integer part. 

Recently, a number of non-scalarization-based methods have been proposed. Forget et al. in \citeyear{forget2022warm} and \citeyear{forget2022branch} focuses on 0-1 problems and employs outer approximation for computing the linear relaxation to generate lower bound sets using a Benson-like algorithm. 
Most other non-scalarization-based algorithms are focused on the more general multiobjective mixed integer convex and nonconvex optimization settings with any number of objectives. 
In particular, the MOMIX~\citep{de2020solving} and HyPaD~\citep{eichfelder2021hybrid,eichfelder2021implementation} algorithms address multiobjective mixed integer convex optimization problems, while the algorithms proposed in~\citep{eichfelder2021general,eichfelder2022solver,link2022computing} address both the convex and nonconvex cases. All of these algorithms compute some kind of enclosure of the EF, which is iteratively refined until a convergence criterion is achieved. 
 The RVF algorithm proposed herein similarly maintains a kind of enclosure, maintaining a global upper bounding function, but the upper bounding function converged to the EF without the need for a lower bounding function. On the other hand, this is why the RVF algorithm requires the solutions of a mixed integer nonlinear optimization problem in each iteration. Because of their focus on the general nonlinear setting, the alternatives mentioned above also have a convergence criterion based on error bounds and do not produce the same kind of discrete representation of the EF that the RVF produces when applied to problems with only linear constraints. 

More recently,~\citet{dunbar2023relaxations} considered the optimization of multiobjective integer programs through various relaxation and duality approaches, including continuous, convex hull, and Lagrangian relaxations, as well as Lagrangian duals and set- and vector-valued superadditive duals, offering alternative solution methods. 
	For more on approaches to approximating the EF using outer approximations, we refer the interested reader to~\citep{benson1998outer,hamel2014benson,csirmaz2016using,lohne2015bensolve, ruzika2005approximation}. A comprehensive survey of branch-and-bound methods for multiobjective linear integer problems can be found in~\citep{przybylski2017multi}. For a detailed review of the literature on multiobjective optimization, we refer the interested reader to~\citep{ehrgott2000survey, Ehrgott2005, ehrgott2016exact}. 
	
	\subsection{Value Function}
	
	The classical VF of an MILP is well-studied, and understanding its structure is crucial for many applications due to its role as a core ingredient in optimality conditions used in a variety of algorithms for solving optimization problems. These optimality conditions are also employed in formulating and solving important classes of multistage and multilevel optimization problems, in which optimality conditions are embedded as constraints in a larger optimization problem. Additionally, optimality conditions are also the basis for techniques used for warm-starting and sensitivity analysis, which are the areas in which the connection to multiobjective optimization is most apparent. 
	
	There have been several studies investigating the structure of the VF in MILPs. \citet{blair1977value} and~\citet{blair1979value} identified fundamental properties of the VF, including that it is comprised of a minimum of a finite number of polyhedral functions. \citet{blair1982value} showed that the VF of a PILP is a \emph{Gomory function}, which is the maximum of subadditive functions known as \emph{Chv\'atal functions}. \citet{blair1984constructive} extended this result to general MILPs, demonstrating that they are the maximum of Gomory functions. Finally,~\citet{blair1995closed} identified what was then referred to as a ``closed-form'' representation of the MILP VF, the so-called \emph{Jeroslow formula}, though this did not lead to what could be considered a practical representation. \citet{GuzRal07} further studied the properties of the VF as it is related to methods of warm-starting and sensitivity analysis and also suggested a method of construction for the case of an MILP with a single constraint. \citet{ralphs2014value} extended this work by providing further details on the structure and properties of the VF for a general MILP and suggesting a practical representation.
	
	Most methods for constructing the VF have focused on the case of PILP, where the discrete structure is the most evident and finite representation is the easiest to achieve. \citet{wolsey1981integer} used a cutting-plane method to derive a sequence of Chv\'atal functions that leads to constructing the full VF for a PILP. \citet{conti1991buchberger} employed reduced Gr\"{o}bner bases and modified the classical Buchberger's algorithm to solve PILPs. Later,~\citet{schultz1998solving} used Gr\"{o}bner basis methods to solve two-stage stochastic programs with complete integer recourse and different RHSs. The authors identified a countable set known as the candidate set of the first-stage variables in which the optimal solution is contained. Then~\citet{ahmed2004finite} developed a global optimization algorithm for solving general two-stage stochastic programs with integer recourse and discrete distributions by exploiting the structure of the second-stage integer problem VF. The authors demonstrated that their algorithm avoids enumerating the search space. \citet{kong2006two} considered a two-stage PILP and presented a superadditive dual formulation that exploits the VF in both stages, solving that reformulation by a global branch-and-bound or level-set approach. \citet{trapp2015note} proposed a constraint aggregation-based approach to alleviate the memory requirement for storing the VF. \citet{zhang2021bilevel} first generalized the complementary slackness theorem to bilevel IP (BIP) and showed that it can be an advantage for constructing the VFs of BIP. The authors also demonstrated that the VFs of BIPs can be constructed by bilevel minimal RHS vectors and presented a dynamic programming algorithm for constructing the BIP VF. Finally, \citet{brown2021gilmore} used a Gilmore-Gomory approach to construct the IP VF.
	
	There have been relatively few algorithmic advances in finding the VF of a general MILP. \citet{bank1982non} studied the qualitative and quantitative stability properties of mixed integer multiobjective optimization problems, which can also be considered an MILP VF. \citet{guzelsoy2006value} proposed algorithms for constructing the VF of an MILP with a single constraint. The properties of the VF and a method for constructing the VF in the case of a general MILP were discussed in~\citet{ralphs2014value}. In the current work, we generalize the work in~\citet{ralphs2014value} to the multiobjective setting. Our hope is that methods of constructing the VF will now quickly be advanced by the adoption of techniques developed in the multiobjective community. 
	
	
	\section{The Restricted Value Function}
	\label{sec:MILP_VF}
	
	Before presenting the main result detailing the relationship between the \VF{} and the EF in Section~\ref{sec:Relationship}, we describe here some basic properties of the \VF{}. In doing this, we hope to demonstrate that certain properties of the EF arise more naturally by considering the structure of the \VF{}. For example, the gradients and subdifferentials of the \VF{} provide a means to understand the trade-offs between objectives in a local region of the EF and we show that these can be fully characterized by considering the dual feasible region of the so-called \emph{continuous restriction}, whose own value function is the so-called \emph{restricted LP value function} (\LPVF{}), described next in Section~\ref{sec:RLPVF}. In Section~\ref{sec:RVF}, we show that the \VF{} can be described as the minimum of a set of translations of the \LPVF{} and characterize its gradients and subdifferentials. 
	
	\subsection{Structure of the Restricted LP Value Function}
 \label{sec:RLPVF}
	
	The \LPVF{} is the special case of the \VF{} when there are no integer variables ($r = 0$). We define the \LPVF{} $\zLP: \Re^{\ell} \rightarrow \Re \cup \{\pm \infty\}$ by
	\begin{mysubequations}{RLPVF} \label{eq.RLPVF_ch2}
		\begin{align}
			\zLP(\zeta) = \inf \;\; & c^0_C x_C\\
		\text{s.t.} ~~	&C_C^{1:\ell} x_C \leq \zeta \label{eq.parametric_ch2} \\
			&A_C x_C = \beta \label{eq.fixed_ch2} \\
			&x_C \in \Re_+^{n - r},
		\end{align}
	\end{mysubequations}
	for all $\zeta \in \Re^{\ell}$. As with the RVF, we define $\zLP$ to take values in the extended reals rather than restricting it to its finite domain, which affects the analysis in significant ways that we discuss below. In this section, we consider $\beta \in \Q^{m}$ to be a fixed vector throughout, while in Section~\ref{sec:RVF}, we consider a parametric class of functions of this form, with different values of $\beta$ arising from fixing the integer part of the solution in~\eqref{eq.RVF_ch2}. As previously, we assume that $\zLP$ is bounded from below.

	The structure of the classical LP value function, where the entire right-hand side is parametric (\( m = 0 \)), is well-studied. It is well established that in this special case, the function $\zLP$ is a polyhedral function whose epigraph is a polyhedral cone described by facets associated with the extremal elements of the dual feasible region $\PD$, defined below. In the more general case of the \LPVF{}, the function is instead a \emph{slice} of this full LP VF, and its epigraph is hence the intersection of a hyperplane with a polyhedral cone. 
	
	In order to analyze the structure of this function, we consider the dual of the LP that arises in the evaluation of $\zLP(\hzeta)$ for $\hzeta \in \Re^{\ell}$, which is
	\begin{equation} \label{eq.D_RLP} \tag{D-RLP}
		\sup_{(u,v) \in \PD} {\hzeta}^{\top} u + \beta^{\top} v,
	\end{equation}
	where $u$ is the vector of dual variables associated with the parametric constraints~\eqref{eq.parametric_ch2}, $v$ is the vector of dual variables associated with the nonparametric constraints~\eqref{eq.fixed_ch2}, and the feasible region is
	\begin{equation*}
		\PD = \left\{(u,v) \in \Re_{-}^{\ell} \times \Re^{m}\;: \; {C_C^{1:\ell}}^{\top} u + A^{\top}_C v \leq c^0_C\right\}.
	\end{equation*}
	By assumption, $\PD$ is nonempty, since $\zLP$ is bounded below. Note that the feasible region $\PD$ of~\eqref{eq.D_RLP} does not depend on $\zeta$. 
	
	Next, let $\Extr$ and $\R$ be the sets of extreme points and extreme rays of $\PD$, respectively. Recall that $\R$ represents the set of extreme elements of the recession cone
	$$
	 \left\{(e, h) \in \Re^{\ell} \times \Re^{m}  \;:\;   {C_C^{1:\ell}}^{\top} e + A^{\top}_C h \leq 0\right\}.
	$$
	By Farkas' lemma~\citep{farkas1902theorie}, $\zLP(\zeta)$ is finite (the problem~\eqref{eq.D_RLP} has an optimal solution) if and only if $\zeta \in \CLP$, where
	$$
	\CLP = \left\{\zeta\in\Re^{\ell}\;:\; \zeta^{\top} e + \beta^{\top} h \leq 0 \;\;  \forall (e,h) \in \R \right\}.
	$$
	Otherwise, we have $\zLP(\zeta) = +\infty$. Note that $\CLP$ is a polyhedron. 
	
	Putting all of this together, we have the following proposition characterizing $\zLP$ and yielding a finite combinatorial description. 
	\bproposition
	\label{epiGraphConvex}
	$\zLP$ is a polyhedral function over $\CLP$ and we have that
	\begin{equation} \label{eq.D_RLPVF} \tag{D-RLPVF}
		\zLP(\zeta) := \max_{(u,v) \in \PD} ({\zeta}^{\top} u + \beta^{\top} v) = 
		\max\limits_{(u, v) \in \Extr} \; ({\zeta}^{\top} u + \beta^{\top} v), \qquad \forall \zeta\in \CLP.
	\end{equation}
	\eproposition
	\begin{proof}
		By~\eqref{eq.D_RLPVF}, the epigraph of \LPVF{} over $\CLP$ is
		\begin{equation*}
			\left\{(\zeta, w)\in \CLP\times\Re\;: \; \zeta^{\top} u +  \beta^{\top} v  \leq w,\ \forall (u, v) \in \Extr \right\},
		\end{equation*}
		which is a polyhedron ($\Extr$ must be finite, since it is the set of extreme points of a polyhedron). Therefore, $\zLP$ is a polyhedral function over $\CLP$. The characterization~\eqref{eq.D_RLPVF} follows from strong duality and the well-known fact that when an LP has a finite optimum, that optimum is achieved at an extreme point. 
	\end{proof}

For an illustration, see Appendix~\ref{RLPVF}, where we detail the \LPVF{} obtained by setting $(x_1, x_2)=(0, 0)$ in Example~\ref{ex.MILPVF}. Since all polyhedral functions are convex, it follows that $\zLP$ is convex over $\CLP$. 

In the remainder of this section, we characterize the directional derivatives and the subdifferentials of the RLPVF. For the classical VF associated with an LP, it is well-known that the subdifferential at a given right-hand side vector within the finite domain of the VF is the set of all optimal solutions to the dual of the LP associated with that given right-hand side. A similar result can be obtained for the RLPVF by projecting this set of optimal dual solutions onto the subspace of the dual variables associated with the parametric constraints~\eqref{eq.parametric_ch2}. The projection will also be useful in characterizing the directional derivatives.

We start by describing the face of all optimal solutions of the LP dual~\eqref{eq.D_RLP} for a given $\zeta\in \CLP$. By adding an optimality constraint to the constraints of $\PD$, we get that the set of optimal solutions with respect to objective vector $(\zeta, \beta)$ is the face
\begin{equation}
OPT(\zeta) = \{(u, v) \in \Re^{\ell}_- \times \Re^m \;:\;  {C_C^{1:\ell}}^{\top} u + A^{\top}_C v \leq c^0_C ,~ \zeta^{\top} u + \beta^{\top} v = \zLP(\zeta)  \} 
\end{equation}
 of $\PD$ (itself a polyhedron). It is well-known how to project a polyhedron into a subspace (see, e.g., Theorem 3.46 of~\cite{ConCorZam14}). Applying this procedure, we get   
\begin{equation} \label{eq:dualoptfaceconvextr} 
		\begin{aligned} 
			\popt{\zeta} 
   &  = \{u \in \Re^{\ell}_- \;:\; \exists v \in \Re^m \textrm{ s.t. } {A^{\top}_C v \leq c^0_C - C_C^{1:\ell}}^{\top} u,~ \beta^{\top} v = \zLP(\zeta) - \zeta^{\top} u \} \\
   &  = \{u \in \Re^{\ell}_- \;:\; (C_C^{1:\ell}r + s \zeta)^\top u \leq r^{\top} c^0_C + s \zLP(\zeta), \; \forall (r, s) \in \mathcal{Q} \},
		\end{aligned}
	\end{equation}
where $\mathcal{Q}$ are the extreme rays of $\{(x_C, t) \in \Re^{n-r}_+ \times \Re \;:\; A_C x_C + t \beta = 0 \}$. 

Note that for $(\hat{x}_C, \hat{t})$ with $\hat{t} < 0$, we can scale the resulting ray so without loss of generality, $\hat{t} = -1$ and then $A_C \hat{x}_C = \beta$, so that $\hat{x}_C$ is an extreme point of $\{x_C \in \Re^{n-r}_+ \;:\; A_C x_C  = \beta \}$, the feasible region of the relaxation obtained by relaxing~\eqref{eq.parametric_ch2}. Then we can re-write the condition for $u$ to be in the projection as
\begin{equation}
(c^0_C - {C_C^{1:\ell}}^{\top} u)^\top r \geq \zLP(\zeta) - \zeta^\top u,
\end{equation}
which has an interpretation in terms of the reduced costs associated with $u$. Note that the above implies that if we precompute the set $\mathcal{Q}$, we can easily obtain a description of $\popt{\zeta}$ for any $\zeta \in \R^\ell$. 

We now state the main result of this section, which is that the set $\popt{\zeta}$ is precisely the subdifferential of the \LPVF{} at $\zeta \in \CLP$. We apply the standard definition for the subdifferential of a convex function~\citep{rockafellar1997convex}, which states that the subdifferential of $\zLP$ at $\hzeta \in \CLP$ is\footnote{This is obviously equivalent to $\partial \zLP(\hzeta) = \{g\in\Re^{\ell}\;:\;
		\zLP(\zeta) \geq g^{\top}(\zeta - \hzeta) + \zLP(\hzeta),\ \forall \zeta \in \Re^{\ell}\}$,
		since $\zLP(\zeta)=+\infty$ for $\zeta \not\in \CLP$.}
	$$
	\partial \zLP(\hzeta) = \left\{g\in\Re^{\ell}\;:\;
	\zLP(\zeta) \geq g^{\top}(\zeta - \hzeta) + \zLP(\hzeta),\ \forall \zeta\in \CLP \right\}.
	$$

	\begin{proposition} \label{prop:subdiffeqdualopt}
		For all $\zeta \in \CLP$,
		$$
		\partial \zLP(\zeta) = \popt{\zeta}.
		$$
	\end{proposition}
 
The proof relies on Lemmas~\ref{lem:dualoptinsubdiff_restricted}, \ref{lem:epi_z_lp}, and \ref{lem:subdiffindualopt_restricted}, all stated and proven in Appendix~\ref{appendix-proofs-sec3}. A consequence of this result is that $\zLP$ is differentiable at points in the interior of its domain at which the dual LP has a unique solution. We formalize this in Corollary~\ref{cor-sec3} in Appendix~\ref{appendix-proofs-sec3}.

We now examine the properties of the directional derivatives to highlight how the function behaves at the boundaries of the finite domain $\CLP$. From the definition of $\CLP$, the following characterization of its interior and boundary can be easily derived. 
	\begin{itemize}
		\item $\zeta \in \inter \CLP$ if and only if $\zeta^\top e + \beta^\top h < 0 \quad \forall (e, h) \in \R.$
		
		\item  $\zeta \in \CLP$ is on the boundary of $\CLP$ if and only if $\zeta^\top e + \beta^\top h \leq 0 \quad \forall (e, h) \in \R$ and there exists $(\hat{e}, \hat{h}) \in \R$ such that $\zeta^\top \hat{e} + \beta^\top \hat{h}  = 0$.
		
		\item  $\zeta \not\in \CLP$ if and only if there exists $(\hat{e}, \hat{h}) \in \R$ such that $\zeta^\top \hat{e} + \beta^\top \hat{h} > 0$.
	\end{itemize}
	
	Now consider a given $\hzeta$ on the boundary of $\CLP$. By the characterization above, there exists $(\hat{e}, \hat{h}) \in \R$ such that $\hzeta^\top \hat{e} + \beta^\top \hat{h}  = 0$. As an aside, it is interesting to observe that this means that specifically for points on the boundary of $\CLP$, we have $\zLP(\hzeta) < +\infty$ (there is a finite optimal value), while  the set of optimal solutions to the LP~\eqref{eq.D_RLP}
 is unbounded. This is because the rays of the optimal face have an objective value of zero.
 
 Let $d \in \Re^{\ell}$ be such that $d^\top \hat{e} > 0$. Then we have that $\hzeta + \epsilon d \not\in \CLP$ for all $\epsilon > 0$, since $(\hzeta + \epsilon d)^\top \hat{e} + \beta^\top \hat{h} > 0$. Thus, we can interpret $d$ as a direction pointing out of $\CLP$. Per the above discussion, we define the set 
$$
	\delta^-(\zeta) =  \cone(\{d \in \Re^\ell \;:\; \exists (e, h) \in \R \textrm{ such that } \zeta^\top e + \beta^\top h  = 0, \; d^\top e > 0\}) \setminus \{\mathbf{0}\},	
$$
	to be the set of all directions pointing out of $\CLP$ at $\zeta \in \CLP$. Note that with this definition, we have $\delta^-(\zeta) = \emptyset$ for $\zeta \in \inter \CLP$, as expected. 
	
	In what follows, we consider several results characterizing the directional derivatives of both the \LPVF{} and the \VF{}. For general $f: \Re^n \rightarrow \Re$, we take the directional derivative of $f$ at $\bar{x}$ in direction $d$ to be
	$$
	\nabla_d f(\bar{x}) = \lim_{t \searrow 0} \frac{f(\bar{x} + t d) - f(\bar{x})}{t}.
	$$
	For both the \LPVF{} and the \VF{} considered in the next section, this limit may go to $+\infty$ at points of discontinuity and we take the directional derivative to have the value $+\infty$ in such cases. For $\zLP$, we have continuity over the finite domain $\CLP$, but discontinuities at points on the boundary of $\CLP$, since we define $\zLP$ over the extended reals. Then the directional derivative $\nabla_d \zLP(\zeta)$ of $\zLP$ at $\zeta \in \CLP$ in direction $d$ is finite if and only if $d \not\in \delta^-(\zeta)$. For $\zeta\in\inter \CLP$, $\delta^-(\zeta) = \emptyset$ and the directional derivative is finite in all directions. 
	
	From the properties of convex functions and subdifferentials, we can alternatively characterize the directional derivative as
	$$
	\nabla_d \zLP(\zeta) = \max_{u\in\partial\zLP(\zeta)} u^{\top}d,
	$$
	where $\partial\zLP(\zeta)$ denotes the subdifferential of $\zLP$ at $\zeta \in \Re^{\ell}$. As a result, we have that for $\zeta \in \CLP$, we have that
	$$
	\nabla_d \zLP(\zeta) = 
	\begin{cases} 
		\max_{u\in\popt{\zeta}} u^{\top}d, & \textrm{if } d \not\in \delta^-(\zeta), \\
		+\infty, & \textrm{otherwise},
	\end{cases}
	$$
	for all $d \in \Re^{\ell}$. Proposition~\ref{prop:differentiable_VF} in Appendix~\ref{appendix-proofs-sec3}	formally establishes this from first principles, though it can also be seen as a corollary of our more general characterization of the subdifferentials of the \LPVF{}, which is the main result of this section. 

	\subsection{Structure of the Restricted Value Function}
	\label{sec:RVF}
	
	We now characterize the structure of the \VF{} by observing that the \VF{} is the minimum of a finite number of translations of functions of the form~\eqref{eq.RLPVF_ch2} for different values of the (previously) non-parametric RHS $\beta$. Each of these translated functions defines a \emph{stability region} over which the integer part of all solutions defining points on the graph of the associated RLPVF is fixed.
	
	To further develop our characterization of the VF, we define the following sets of integer parts of solutions by projecting $\S(\zeta)$ onto the space of the integer variables:
	\begin{equation*}
		\S_I(\zeta) = \proj_I \S(\zeta) = \left\{x_I\in \Z_+^r\;: \; (x_I, x_C) \in \S(\zeta)\right\}, \text{ and}
	\end{equation*}
	\begin{equation*}
		\S_I = \bigcup_{\zeta \in \C} \S_I(\zeta).
	\end{equation*}
	Thus, $\S_I$ is the set of all integer parts of points in $\S(\zeta)$ for some $\zeta \in \C$.
	For a given $\hat{x}_I \in \S_I$, the \textit{continuous restriction} (CR) with respect to $\hat{x}_I$ is the function\footnote{In the multiobjective optimization literature, the corresponding multiobjective LP is sometimes referred to as a {\em slice problem}~\citep{belotti2013branch}.}
	\begin{equation*}
		\bar{z}(\zeta;\ \hat{x}_I) = c^0_I \hat{x}_I + \zLP(\zeta - C_I^{1:\ell} \hat{x}_I;\ b - A_I \hat{x}_I) \quad \forall \zeta \in \C, \tag{CR}\label{eq.CR_ch2}
	\end{equation*}
	which is defined in terms of the previously defined function $\zLP$, but with an additional secondary parameter. In particular, $\zLP(\cdot;\ \beta)$ is similar to the previously defined $\zLP$ except that we now allow for a parametric family of such functions with different vectors for the non-parametric RHS $\beta$ in~\eqref{eq.RLPVF_ch2}. The form shown above is precisely a translation of a function of the form~\eqref{eq.RLPVF_ch2} for $\beta = b - A_I \hat{x}_I$. In the remainder of the study, we refer to functions $\bar{z}(\cdot;\ x_I)$ for $x_I \in \S_I$ as \emph{bounding functions}, since they bound the RVF from above, as demonstrated in Proposition~\ref{z_bar_bound}.
	
	\bproposition
	\label{z_bar_bound}
	For any $\hat{x}_I \in \S_I$, $\bar{z}(\cdot;\ \hat{x}_I)$ bounds $z$ from above.
	\eproposition
	\begin{proof}
		For $\hat{x}_I \in \S_I$, we have
		\begin{equation*}
			\begin{array}{ll@{}ll}
				\bar{z}(\zeta;\ \hat{x}_I) = c_I^0 \hat{x}_I + &\zLP(\zeta - C_I^{1:\ell} \hat{x}_I;\ b - A_I \hat{x}_I) \geq\\
				&\min\limits_{x_I \in \S_I} \left(c_I^0 x_I + \zLP(\zeta - C_I^{1:\ell} x_I;\ b - A_I x_I)\right) = z(\zeta), \quad \forall \zeta \in \C.\\
	\end{array}\end{equation*}\end{proof}

	Proposition~\ref{z_bar_bound} shows that any collection of points from $\S_I$ yields an upper approximation of $z$ simply by taking the minimum of the associated set of bounding functions, previously defined as $\bar{z}(\cdot;\ x_I)$ for $x_I \in \S_I$. The algorithm described in Section~\ref{sec:FiniteAlgorithm} constructs a subset of $\S_I$ that fully describes the RVF by iteratively approximating it from above. 
 
We first focus on characterizing the \VF{} as the minimum of a set of bounding functions, as defined in Proposition~\ref{z_bar_bound}.
	When $\XMO$ is bounded, $S_I$ is finite, so the number of such functions required is finite. 
As such, the main result of this section is the following discrete characterization.
	
	\btheorem
	\label{thm:Dis_Representation}
	Let $\S$ be any subset of $\S_I$ with the property that for any $\zeta \in \C$, there exist $x_I \in \S$ and $x_C \in \Re_+^{n - r}$ such that $C_I^{1:\ell} x_I + C_C^{1:\ell} x_C \leq \zeta,\ A_I x_I + A_C x_C = b$, and $c_I^0 x_I + c^0_C x_C = z(\zeta)$. Then for any $\zeta \in \C$ we have 
	\begin{equation*}
		z(\zeta) = \min_{x_I \in \S_I} \bar{z}(\zeta;\ x_I) = \min_{x_I \in \S} \bar{z}(\zeta;\ x_I).
	\end{equation*}
	\etheorem
	\begin{proof} Let $\zeta\in\C$. Since $\S \subseteq\S_I$, then by Proposition~\ref{z_bar_bound}, we have that
		\begin{equation*}
			z(\zeta) \leq \min_{x_I \in \S} \bar{z}(\zeta;\ x_I).
		\end{equation*}
		Now by definition of $\S$, there exists $\hat{x}_I\in\S$ and $\hat{x}_C \in \Re_+^{n - r}$ such that $C_I^{1:\ell} \hat{x}_I + C_C^{1:\ell} \hat{x}_C \leq \zeta,\ A_I \hat{x}_I + A_C \hat{x}_C = b$, and $c_I^0 \hat{x}_I + c^0_C \hat{x}_C = z(\zeta)$. So
		\begin{eqnarray*}
			\min_{x_I \in \S} \bar{z}(\zeta;\ x_I) & = &
			\min_{x_I \in \S} \left(c^0_I x_I + \zLP(\zeta - C_I^{1:\ell} x_I;\ b - A_I x_I)\right) \\
			& \leq & c^0_I \hat{x}_I + \zLP(\zeta - C_I^{1:\ell} \hat{x}_I;\ b - A_I \hat{x}_I)\\
			& \leq & c^0_I \hat{x}_I  + c^0_C \hat{x}_C = z(\zeta),
		\end{eqnarray*}
		where the first equation follows from~\eqref{eq.CR_ch2}, the subsequent inequality follows since $\hat{x}_I\in\S$, the next since, by its definition, $\hat{x}_C$ is feasible for $\zLP(\zeta - C_I^{1:\ell} \hat{x}_I;\ b - A_I \hat{x}_I)$, and the final equation follows from the definition of $\hat{x}_I$ and $\hat{x}_C$. We thus have both that $z(\zeta)$ is no greater and no less than $\min_{x_I \in \S} \bar{z}(\zeta;\ x_I)$, and the result follows.
	\end{proof}
We call a set $\S$ with the properties described in the above theorem a \emph{description} of the RVF. Of course, what constitutes a full ``description'' of this function is largely a philosophical question. Given the complexity of the function itself, there is obviously a trade-off between the time needed to construct a given description and the ease with which information about the function can be extracted. Determining the ``best'' description is itself a multiobjective problem, The description proposed here allows for $z$ to be evaluated in time polynomial in the size of the original problem data and $\lvert \S \rvert$. We also show below that with a little additional computation, we can also describe the function's gradients and subdifferential (the latter are polyhedra). 

It is clearly desirable to construct a description that is minimal with respect to the properties of Theorem~\ref{thm:Dis_Representation}. Although such a minimal description is clearly not necessarily unique, we denote any such a minimal description by the notation $\Smin$, since the particular one chosen does not affect our results. Each member $\hat{x}_I\in\Smin$ is associated with its own \emph{stability region}, denoted by $\C(\hat{x}_I)$, and defined to be
	\begin{equation*}
		\C(\hat{x}_I) = \left\{\zeta\in\C \;: \; z(\zeta) = \bar{z}(\zeta;\ \hat{x}_I)\right\}.
	\end{equation*}
	Informally, this is the subset of $\C$ for which the bounding function associated with $\hat{x}_I$ agrees with the \VF{}. We discuss the properties of stability regions below, but generally, stability regions do not have to be closed or connected, as exemplified in the following example.
	
	\bexample
	The stability regions in Example~\ref{ex.MILPVF} are depicted in Figure~\ref{Fig:stabRegions} for the domain $\zeta \geq -15$. Here, $\S_I = \Smin = \{(0,0), (0,1),(1,0),(1,1)$. 
	\begin{figure}[tb]
		\centering
		\includegraphics[width=.6\linewidth]{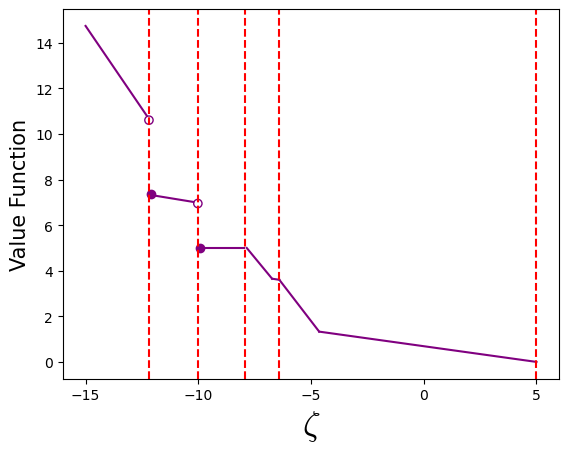}
		\caption{Stability regions and corresponding members of $\S_I$ in Example~\ref{ex.MILPVF}, for the region $\zeta\geq -15$, stability regions from left to right correspond to $(x_1,x_2) \in \{(1,0), (1,1), (0,1), (1,0), (0,0)\}.$ }\label{Fig:stabRegions}
	\end{figure}
The stability regions are
\begin{equation*}
\begin{aligned}
    \C((1,1)) &= [-12.167,-10) \\
    \C((0,1)) &= [-10,-7.833] \\
    \C((1,0)) &= [-55.464 , -12.167) \cup [-7.833, -6.4] \\
    \C((0,0)) &= [-6.4,+\infty). \\
\end{aligned}
\end{equation*} 
Note that the stability region for $(1,0)$ consists of two disjoint intervals, one of which is open. The stability region for $(0,0)$ is a single interval, closed at both ends. The stability region for $(0,1)$ is a single closed interval on which the \VF{} has a zero gradient. Further insights can be obtained by examining the bounding functions associated with each element of $\S_I$, provided in Appendix~\ref{appendix-whole_EF}.
	\eexample
	
Even when connected, a stability region may be nonconvex and open, as illustrated in Example~\ref{notConvexNotCompact}. 
	
	\bexample\label{notConvexNotCompact}
	Consider the RVF instance given by
	\begin{equation*}
\begin{aligned}
    z(\zeta) = \min \quad  &\,x_2 \\
    \emph{s.t.}\quad & x_3 \leq \zeta_1\\
    & x_4 \leq \zeta_2\\
    &2(1-x_1) \leq x_j \leq 2(1-x_1) + 5x_1, \quad j=2,3,4\\
    & x_2 + x_3 \geq 5x_1 \\
    & x_1 \in \{0,1\}\\
    & x_j \leq 5, \quad x_{j} \in \mathbb{R}_+, \quad j=2,3,4.
\end{aligned}
\end{equation*}\label{MILP_3D_instance}
	As in previous examples, although this is not in the standard form as in~\eqref{eq.RVF_ch2}, it can easily be made so with the addition of slack and surplus variables. Since these do not change the \VF{} or stability regions, we keep the instance in its natural form. The VF for this instance can be written explicitly as
	$$
	z(\zeta) = \left\{ \begin{array}{ll}
		2, & \zeta_1 \in [2,3] \text{ and } \zeta_2\geq 2, \\
		5 - \zeta_1, & \zeta_1\leq 5 \text{ and }\zeta_2 < 2 \text{ or } \zeta_1\in [0,2)\cup(3,5], \\
		0, & \text{otherwise},
	\end{array}\right.
	$$
	for $\zeta = (\zeta_1, \zeta_2) \in {\cal C} = \Re_+^2$. The stability region for $x_I=(x_1)=(1)$ is 
	$$
	\{\zeta\in \Re_+^2\;:\; \zeta_1 < 2 \text{ or } \zeta_1 \geq 3 \text{ or } \zeta_2 < 2\},
	$$
	which is shown in Figure~\ref{Fig:stabregion_nonconvex}, which displays both stability regions for this example. The stability region for $x_I=x_1=1$ is connected but is neither convex nor closed. 
	\eexample
	
	\begin{figure}[htbp]
		\centering
		\includegraphics[width=.5\linewidth]{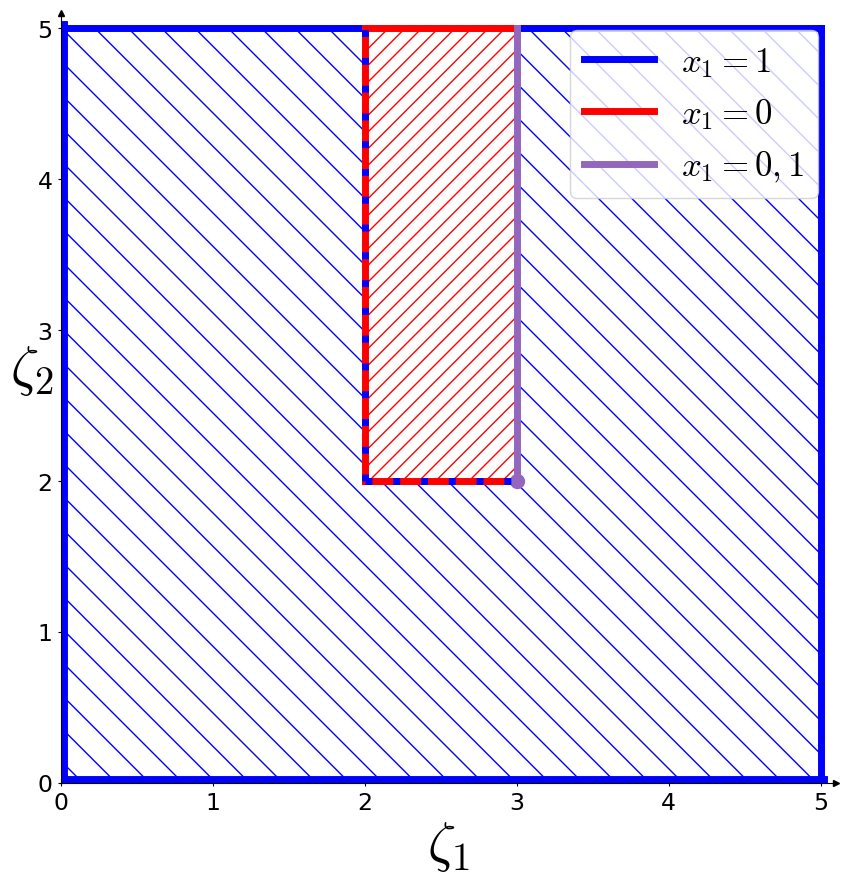}
		\caption{Stability regions of the instance in  Example~\ref{notConvexNotCompact}. The area shaded in red, including the red boundary lines, is not in the stability region for $x_1=1$, which is thus neither convex nor closed. The line segment shaded in purple is in the stability region of both $x_1=0$ and $x_1=1$.}\label{Fig:stabregion_nonconvex}
	\end{figure}
 
	In general, we do not expect there to be a minimal description that is unique; it is possible for multiple integer solutions $x^1_I, x^2_I \in \S_I$ with $x^1_I \not= x^2_I$ to satisfy $\bar{z}(\cdot;\ x^1_I) = \bar{z}(\cdot;\ x^2_I)$, e.g., in the trivial case of duplicate variables. Furthermore, for any point $(\zeta,z(\zeta))$ on the graph of the \VF{}, there may be more than one integer part $x_I$ with this point lying on the \VF{} for its associated bounding function. In other words, it is possible that stability regions may overlap, even for integer parts in a given minimal description. Such a case occurs in Example~\ref{ex.MILPVF}: points in the purple-colored line segment seen in Figure~\ref{Fig:stabregion_nonconvex} are in the stability regions of two different integer parts, both of which are in $\Smin$. Of course, for each integer part $x_I\in\Smin$, there must exist at least {\em one} point on the boundary of the epigraph of the associated bounding function that is on the boundary of the epigraph of the \VF{} itself and is not on the boundary of the epigraph of any other bounding function; otherwise, $\Smin$ would not be minimal. All we can say for sure is that elements of $\S_I$ that have an empty stability region will not be present in any minimal description. 
	
	We now turn our attention from the structure of stability regions and properties of minimal descriptions to properties of the \VF{} itself arising from the finite description of Theorem~\ref{thm:Dis_Representation}. Figure~\ref{Fig:vf_first_instance} illustrates the \VF{} of the MILP described in Example~\ref{ex.MILPVF}. The \VF{} is piecewise polyhedral, nonincreasing, and lower semi-continuous (here, nonincreasing means that $\zeta_1 \geq \zeta_2,~\zeta_1, \zeta_2 \in \Re^{\ell} \Rightarrow z(\zeta_1) \leq z(\zeta_2)$). We formalize this in Proposition~\ref{MVF_Pro}, which can be viewed as a generalization of the previous results by~\citep{nemhauser1988scope} and~\citep{bank1982non} for the full MILP VF. We articulate this within Proposition \ref{MVF_Pro} found in Appendix \ref{appendix-proofs-sec3}.

	The \VF{} may be discontinuous due to the fact that individual bounding functions themselves are discontinuous at the boundaries of their finite domains. When the \VF{} is discontinuous, the discontinuities occur at (some of) the boundaries between stability regions. At such points, as well as at points on the boundary of $\C$, there may be directions in which the directional derivative of $z$ is infinite, exactly as can occur at points on the boundary of $\CLP$ in the RLPVF. 
	
	Next, we describe the directional derivative of $z$ at a point $\zeta\in\C$ in terms of the faces of optimal solutions to the duals of the LPs associated with the bounding functions in directions for which the directional derivative is finite. To characterize these directions, we need the analogues of the sets of optimal extreme points and extreme rays characterizing the optimal faces of solutions to the LP dual, as we had in the RLPVF case, but with the integer part of the solution fixed. For $x_I \in \S_I$ and $\zeta \in \Re^{\ell}$, we define 
	\begin{equation} 
		\begin{aligned} 
			\popt{\zeta; x_I} 
   &  = \{u \in \Re^{\ell}_- \;:\; \exists v \in \Re^m \textrm{ s.t. } {A^{\top}_C v \leq c^0_C - C_C^{1:\ell}}^{\top} u,~ \beta^{\top} v = \bar{z}(\zeta;\ x_I) - (\zeta - C_I^{1:\ell} x_I)^{\top} u \} \\
   &  = \{u \in \Re^{\ell}_- \;:\; (C_C^{1:\ell}r + s (\zeta - C_I^{1:\ell} x_I))^\top u \leq r^{\top} c^0_C + s \bar{z}(\zeta;\ x_I), \; \forall (r, s) \in \mathcal{Q} \},
		\end{aligned}
	\end{equation}
where $\mathcal{Q}$ is defined as previously to be the extreme rays of $\{(x_C, t) \in \Re^{n-r}_+ \times \Re \;:\; A_C x_C + t \beta = 0 \}$.	We further need to know the set of integer vectors associated with the bounding functions that agree with the \VF{} at a given point, which we denote as
	\begin{equation*}
		\begin{array}{ll@{}ll}
			S^*_I(\zeta) &= \left\{x_I\in S_I(\zeta)\;:\; c^0_Ix_I + c^0_C x_C = z(\zeta),\ \exists x_C \text{ s.t. } (x_I,x_C)\in S(\zeta)\right\}\\
			&= \left\{x_I\in S_I(\zeta)\;:\;\bar{z}(\zeta;\ x_I) = z(\zeta)\right\}.
		\end{array}
	\end{equation*}
	With all of this notation established, let us now consider a point $\hzeta \in \Re^{\ell}$ at which the \VF{} is discontinuous. As with the RLPVF, the directional derivative of a bounding function associated with $x_I \in \S_I$ is infinite in direction $d \in \Re^{\ell}$ at $\hzeta$ if and only if $d$ is contained in the set
	\begin{equation*}
		\begin{aligned}
			\delta^-(\hzeta;\ x_I) = \cone&(\{d \in \Re^\ell \;:\; \exists (e, h) \in \R 
			\textrm{ such that } \\ 
			& (\hzeta - C_I^{1:\ell} x_I)^\top e + (b - A_I x_I)^\top h  = 0,
			\; d^\top e > 0\}) \setminus \{\mathbf{0}\}.
		\end{aligned}
	\end{equation*}
	Finally, for the directional derivative of $z$ itself to be infinite in direction $d \in \Re^{\ell}$ at $\hzeta \in \Re^{\ell}$, the derivatives of \emph{all} bounding functions that agree with $z$ at $\zeta$ must also be infinite, so we have the following proposition, whose proof appears in Appendix~\ref{appendix-proofs-sec3}.
	
	\bproposition \label{RVFDirDev} For $\zeta\in \C$ and $d \in \Re^{\ell}$, 
	\begin{align*}
		&\nabla_d z(\zeta)  = \min_{x_I \in \S_I^*(\zeta) \cap \Smin} \nabla_d \bar{z}(\zeta;\ x_I) \\
		 &= \begin{cases} 
			\min\limits_{x_I \in \S_I^*(\zeta) \cap \Smin} (\max_{u\in\popt{\zeta;\ x_I}} u^{\top} d), & \textrm{if } d \not \in \cap_{x_I \in \S_I^*(\zeta) \cap \Smin} \delta^-(\zeta;\ x_I), \\
			\infty, & \textrm{otherwise},
		\end{cases}
	\end{align*}
	where $\Smin$ is any minimal description of the \VF{}.
	\eproposition
	
	Even though the \VF{} is not convex, nor expected to be continuous, it is still possible to define a notion of subdifferential by considering the local structure of the VF at a given point, which is inherited from the bounding functions that are active at the point and necessary to a minimal description of the VF. 
	
	\begin{definition}[Local Subdifferential]
		Let $f:\Re^{\ell}\rightarrow\Re$ and $\zeta \in \Re^{\ell}$. Then $q\in\Re^{\ell}$ is a \emph{local subgradient} of $f$ at $\zeta' \in \Re^\ell$ if there exists $\epsilon>0$ such that
		\begin{equation}
			\label{eq:localsubgdef}
			f(\zeta) \ge f(\zeta') + q^{\top}( \zeta - \zeta'), \quad \forall \zeta \in \Re^\ell \textrm{ such that } \|\zeta - \zeta'\| \leq \epsilon.
		\end{equation}
		The \emph{local subdifferential} of $f$ at $\zeta$, denoted $\partial_L f(\zeta)$, is defined as the set of all local subgradients of $f$ at $\zeta$.
	\end{definition}
	Our definition of a local subgradient is closely related to that of the proximal subgradient given in~\citep{rockafellar2009variational} (Definition 8.45). The main difference is that we drop the quadratic term $\sigma \|\zeta'-\zeta\|^2$ in the definition of the proximal subgradient, since $\sigma = 0$ is always valid in our setting. In particular, for functions that are piecewise affine or convex, our local subdifferential and the proximal subdifferential are identical. 

    Finally, we are ready to state the main result of this section, that the intersection of the subdifferentials of these active and necessary bounding functions at a point yields the local subdifferential of the VF at that point. The proof of this result depends crucially on the following observation, stated formally as Lemma~\ref{lem:upifinactive} in Appendix~\ref{appendix-proofs-sec3}. Since (by Proposition~\ref{MVF_Pro}) $z$ is comprised of a minimum of a finite number of polyhedral functions, each associated with a member of some minimal description $\Smin$, as per in Theorem~\ref{thm:Dis_Representation}, then $z$ must agree with the upper bounding function associated with \emph{some} member of $\Smin$ in the local neighborhood of any $\zeta \in \C$. The remaining details of the proof appear in Appendix~\ref{appendix-proofs-sec3}.
	
	\bproposition \label{RVFDiff}
	The local subdifferential of $z$ at $\zeta\in\C$ is given by
 $$\partial_L z(\zeta)
	=\bigcap_{x_I\in S^*_I(\zeta)\cap\Smin}
	\popt{\zeta-C^{1:\ell}_Ix_I;\ b-A_Ix_I},$$
	where $\Smin$ is a minimal description of the \VF{}.
	\eproposition

	
	\section{The \VF{} and the EF}
	\label{sec:Relationship}
	We now formalize the relationship between the \VF{} and the EF. Informally, Theorem~\ref{thm:Relationship} below states that the boundary of the epigraph of the \VF{} and the EF effectively encode the same information, with the boundary of the epigraph including weak NDPs, while the  EF does not.
		
	\btheorem
	\label{thm:Relationship}
	The EF associated with~\eqref{eq.MultiObj_ch2} is a (possibly strict) subset of the boundary of the epigraph of the \VF{}. In particular, the following statements hold.
	\begin{enumerate}
		\item \label{relation-1} If $C_I x_I + C_C x_C$ belongs to the EF for some $(x_I,x_C) \in \XMO$, then $(C_I^{1:\ell} x_I + C_C^{1:\ell} x_C, c^0_Ix_I + c^0_Cx_C)$ is a point on the boundary of the epigraph of $z$.
		
		\item \label{relation-2} If $(\zeta, z(\zeta))$ is a point on the boundary of the epigraph of $z$, then there exists an efficient solution $(x_I,x_C) \in \XMO$ such that $c^0_Ix_I + c^0_Cx_C = z(\zeta)$ and $C_I^{1:\ell} x_I + C_C^{1:\ell} x_C \leq \zeta$. Further, $C_I^{1:\ell} x_I + C_C^{1:\ell} x_C = \zeta$ if and only if $\nabla_d z(\zeta) > 0$ for all $d \in \Re^{\ell}_-\setminus\{\mathbf{0}\}$.
	\end{enumerate} 
	\etheorem
	
	\begin{proof}
		\begin{enumerate}
			\item To prove statement~\ref{relation-1}, let $(\hat{x}_I,\hat{x}_C) \in \XMO$ be a given efficient solution and let $\hat{\zeta} = C_I^{1:\ell} \hat{x}_I + C_C^{1:\ell} \hat{x}_C$. Now $(\hat{x}_I,\hat{x}_C) \in \XMO$ and $\hat{\zeta} = C_I^{1:\ell} \hat{x}_I + C_C^{1:\ell} \hat{x}_C$ implies, from definitions, that $(\hat{x}_I,\hat{x}_C) \in \S(\hat{\zeta})$. We want to show that $(\hat{\zeta}, c^0_I\hat{x}_I + c^0_C \hat{x}_C)$ is a point on the boundary of the epigraph of the RVF $z$. Since $z(\hat{\zeta}) = \min_{(x_I, x_C) \in \S(\hat{\zeta})} c^0_I x_I + c^0_C x_C$, by definition, and since $(\hat{x}_I,\hat{x}_C) \in \S(\hat{\zeta})$, it must be that $z(\hat{\zeta})\leq c^0_I\hat{x}_I + c^0_C \hat{x}_C$. Assume, for the sake of contradiction, that $c^0_I \hat{x}_I + c^0_C \hat{x}_C \not= z(\hat{\zeta})$. 
			Then it must be that $c^0_I \hat{x}_I + c^0_C \hat{x}_C > z(\hat{\zeta})$. Now, by the definition of $z(\hat{\zeta})$, there must exist $(x_I, x_C) \in \S(\hat{\zeta})$ with $z(\hat{\zeta})=c^0_I x_I + c^0_C x_C$. Furthermore, $(x_I, x_C) \in \S(\hat{\zeta})$ implies $C_I^{1:\ell} x_I + C_C^{1:\ell} x_C \leq\hat{\zeta}=C_I^{1:\ell} \hat{x}_I + C_C^{1:\ell} \hat{x}_C$, while $c^0_I x_I + c^0_C x_C = z(\hat{\zeta}) < c^0_I \hat{x}_I + c^0_C \hat{x}_C$, so $C_I x_I + C_C x_C \lneqq C_I \hat{x}_I + C_C \hat{x}_C$.  This contradicts the hypothesis that $(\hat{x}_I, \hat{x}_C)$  is an efficient solution, and the result follows.
			
			\item To prove the first part of~\ref{relation-2}, let $\zeta \in \C$ be given, so that $(\zeta,z(\zeta))$ is a point on the boundary of the epigraph of $z$. Then there exists $(\tilde{x}_I,\tilde{x}_C) \in \S(\zeta)$ such that $c^0_I\tilde{x}_I + c^0_C \tilde{x}_C = z(\zeta)$. There are two cases. If $(\tilde{x}_I,\tilde{x}_C)$ is an efficient solution, the result follows trivially. Otherwise, there must exist an efficient solution $(x_I, x_C) \in \XMO$ that dominates $(\tilde{x}_I,\tilde{x}_C)$ at least weakly, i.e., such that $C_I x_I + C_C x_C \leq C_I\tilde{x}_I+C_C\tilde{x}_C$. Then $C^{1:\ell}_I x_I + C^{1:\ell}_C x_C \leq C^{1:\ell}_I\tilde{x}_I+C^{1:\ell}_C\tilde{x}_C \leq \zeta$, which means $(x_I,x_C)\in \S(\zeta)$ and we also have $c^0_Ix_I + c^0_Cx_C\leq c^0_I\tilde{x}_I + c^0_C \tilde{x}_C=z(\zeta)$. But by the optimality of $z(\zeta)$, we must also have that $c^0_Ix_I + c^0_Cx_C \geq z(\zeta)$, so we conclude that $c^0_Ix_I + c^0_Cx_C = z(\zeta)$. Since $(x_I,x_C)$ is an efficient solution, the result follows.

			For the second part, there are two directions. 
			\begin{enumerate}
				\item[$\Leftarrow$] Suppose that $\nabla_d z(\zeta) > 0$ for all $d\in \Re^{\ell}_-\setminus\{\mathbf{0}\}$ and assume for the sake of contradiction that $\zeta \gneqq C_I^{1:\ell} x_I + C_C^{1:\ell} x_C$. Let $\tilde{\zeta} = C_I^{1:\ell} x_I + C_C^{1:\ell} x_C$. Now $\S(\tilde{\zeta})\subseteq\S({\zeta'})\subseteq\S(\zeta)$ for all $\zeta'$ with $\tilde{\zeta}\leq{\zeta'}\leq\zeta$. Thus, since $(x_I,x_C)$ minimizes the objective over $\S(\zeta)$ and $(x_I,x_C)\in\S(\tilde{\zeta})$, it must be that $z(\zeta')=c_I^0 x_I + c_C^0 x_C$ for all $\zeta' \in [\tilde{\zeta},\zeta]$. So $z(\zeta') = z(\zeta)$ for all $\zeta' \in [\tilde{\zeta},\zeta]$. Recalling that $\zeta \gneqq \tilde{\zeta}$, it must thus be that $\hat{d}=\tilde{\zeta}-
				\zeta\in\Re^{\ell}_-\setminus\{\mathbf{0}\}$ and $\nabla_{\hat{d}} z(\zeta) = 0$, which is a contradiction to the initial hypothesis. Therefore, $\zeta = C_I^{1:\ell} x_I + C_C^{1:\ell}x_C$.
				
				\item[$\Rightarrow$] We prove the contrapositive. Therefore, suppose there exists $d \in \Re^{\ell}_-\setminus\{\mathbf{0}\}$ such that $\nabla_d z(\zeta) = 0$. Then, by Proposition~\ref{MVF_Pro}, there must exist $\tilde{\zeta}\lneqq\zeta$ with $z(\tilde{\zeta})=z(\zeta)$. 
				By the proof of the first part of~\ref{relation-2}, above, there must exist $(\tilde{x}_I,\tilde{x}_C)$ an efficient solution with $(\tilde{x}_I,\tilde{x}_C)\in\S(\tilde{\zeta})$ and $c^0_I\tilde{x}_I + c^0_C \tilde{x}_C =z(\tilde{\zeta})$. Now the former condition implies $C^{1:\ell}_I\tilde{x}_I+C^{1:\ell}_C\tilde{x}_C\leq \tilde{\zeta} \lneqq \zeta$ while the latter implies that $c^0_I\tilde{x}_I + c^0_C \tilde{x}_C =z(\zeta)$, since $z(\tilde{\zeta})=z(\zeta)$. The result follows.
			\end{enumerate}
		\end{enumerate}\end{proof}
	
	Since the theorem is somewhat technical, we now further explain the intuition of the result. Part~\ref{relation-1} is a straightforward statement that every point on the EF is also a point on the boundary of the epigraph of the RVF. The technicalities in Part~\ref{relation-2} arise from the aforementioned fact that there are points on the boundary of the epigraph of the RVF that are not in the EF. A point on the boundary of the epigraph that is not contained in the EF corresponds to a weak NDP and can thus be associated with the one or more NDPs that weakly dominate it. 
	
	For illustration, consider $(\zeta, \zLP(\zeta))$ on the boundary of the epigraph of $z$. If the function value strictly increases whenever any component of the argument $\zeta$ decreases, then the directional derivative is positive in all negative directions and this point must correspond to an NDP. 
 For example, in Example~\ref{ex.MILPVF}, $\nabla_d z(-11) > 0$ for $d = -1$ (see Figure~\ref{Fig:stabRegions}) and thus $(-11, z(-11))$ corresponds to an NDP.  
	
	On the other hand, if there is a direction $d$ in which the directional derivative is zero at $\zeta$, then the point $(\zeta, \zLP(\zeta))$ must be a weak NDP, since moving in the direction $d$ from $\zeta$ corresponds to strictly improving the value of one or more of the multiple objectives that correspond to the parametric constraints, while the objective of~\eqref{eq.MILP_ch2} remains unchanged. 
 For example, in Example~\ref{ex.MILPVF}, points in the interior of the stability region associated with $(0, 1)$ (e.g., $\zeta = -9$) correspond to weak NDPs since they are dominated by the point $(-10, 5)$. 
	
	The conditions involving directional derivatives also have another interpretation that is possibly more intuitive. Recall from the previous section that when $\nabla_d z(\hzeta)$ is finite for $\hzeta \in \Re^\ell$, we have that $\nabla_d z(\hzeta) = d^{\top}\hu$ for some optimal solution $\hu \in \PD$ to~\eqref{eq.D_RLP}. This allows us to re-interpret the above conditions involving directional derivatives in terms of solutions to~\eqref{eq.D_RLP}. In particular, the condition $\nabla_d z(\hzeta) > 0$ for all $d \in \Re^{\ell}_-\setminus\{\mathbf{0}\}$ is equivalent to $u < \mathbf{0}$ for all alternative optimal solutions $(u, v) \in \PD$ for~\eqref{eq.D_RLP} associated with $\hzeta$, while a zero directional derivative implies that the dual variable associated with one of the constraints is zero. This makes sense, as a zero dual value implies that the constraint can be tightened without changing the optimal solution, and this is exactly the condition that would indicate a given solution is not nondominated in the multiobjective context. 
	
	Note that when all directional derivatives in directions $d \in \Re^{\ell}_-\setminus\{\mathbf{0}\}$ are strictly positive (the function is strictly decreasing everywhere), then the boundary of the epigraph of the \VF{} and the EF exactly coincide. 
 This case is illustrated in Example~\ref{Ex:discont} below.
	
	\bexample
	We consider the following instance of~\eqref{eq.RVF_ch2}:
	\begin{equation*}
		\begin{aligned}
			z(\zeta) = \min ~ &\ x_1 + \frac{1}{4}x_2 + \frac{1}{2}y_1  -\frac{3}{4}y_2\\
			 \emph{s.t.} ~~  &\frac{4}{5}x_1 + \frac{1}{2}x_2 + \frac{1}{3}y_1 + 0y_2 \leq \zeta\\
			&\frac{3}{5}x_1 + \frac{1}{3}x_2 + \frac{1}{4}y_1  -\frac{1}{5}y_2 = 4\\
			&x_{i} \in \Z_+, \quad \forall i \in \{1, 2\}\\
			&y_{j} \in \Re_+, \quad \forall j \in \{1, 2\}.
		\end{aligned}
	\end{equation*}
	\label{Ex:discont}
	Figure~\ref{Fig:sameEF_RVF} shows the EF in Example~\ref{Ex:discont}, which exactly coincides with the boundary of the epigraph of the associated RVF.
	\eexample
	
	\begin{figure}[htbp]
		\centering
		\includegraphics[width=.6\linewidth]{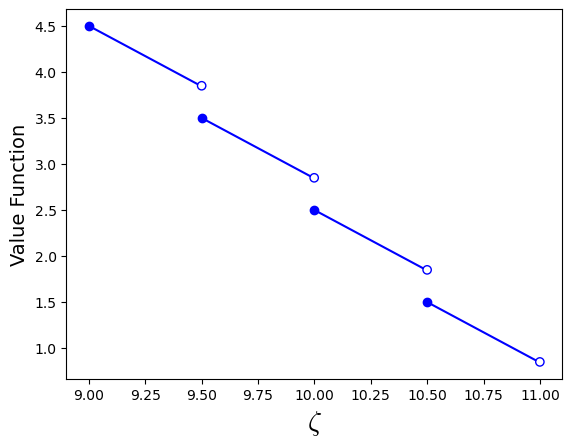}
		\caption{EF and RVF for the MILP in Example~\ref{Ex:discont}} \label{Fig:sameEF_RVF}
	\end{figure}
	
	It is worth noting that it is possible to avoid the difficulty of the weak NDPs that aren't part of the EF by changing from ``$\leq$''  to equality for the constraints associated with the objectives of the multiobjective version of the problem. However, in that case, a different difficulty is introduced---there may then be parts of the \VF{} that are \emph{increasing} (strictly positive directional derivative in the direction $d \in \Re^{\ell}_+$), and we then have that the boundary of the epigraph for \emph{those} parts of the \VF{} is not part of the EF. As such, this approach therefore does not make the statement of the theorem any cleaner. In Appendix~\ref{appendix-relationship_equality}, we state the equivalent Theorem~\ref{thm:Relationship_appendix}, noting that the only deviation from the original statement is replacing $\leq$ with $=$ for the parametric constraints.
	
	Before closing this section, we use the results of Theorem~\ref{thm:Relationship} to connect the description of the \VF{}, from Theorem~\ref{thm:Dis_Representation}, and the EF of the~\eqref{eq.MultiObj_ch2}. Specifically, we will see that not only does a minimal subset $\Smin$ give us a description of the VF but also a description of the EF (under the previous caveat that what constitutes a ``description'' is context dependent). In doing so, we provide some insight into the role of the objective vector $c^0$, which is distinguished in the~\eqref{eq.RVF_ch2} problem but is indistinguishable from other objectives in the~\eqref{eq.MultiObj_ch2}. 
	
	First, 
 suppose that the~\eqref{eq.MultiObj_ch2} has been solved in some way to obtain a set $\F$ that is the set of points on the EF. We wish to determine the \VF{}. The first part of Theorem~\ref{thm:Relationship} says 
	$$
	z(\zeta) = \min_{\gamma\in\F}\left\{\gamma_0\;:\; \gamma^{1:\ell}\leq\zeta\right\}, \qquad \forall \zeta\in\cal C.
	$$
	Alternatively, the~\eqref{eq.MultiObj_ch2} may be solved by providing some set of integer parts of efficient solutions, denoted as $\E_I\subseteq \S_I$ so that for all $\gamma\in\F$ there must exist an efficient solution $(x_I,x_C)\in\XMO$ with $x_I\in\E_I$ and $\gamma=C_Ix_I+C_Cx_C$. Then, by Theorem~\ref{thm:Dis_Representation}, there is a minimal description $\Smin\subseteq \E_I$. 
	
	Now we take the opposite perspective: suppose that we can determine a \VF{} (by finding $\Smin$ as per Theorem~\ref{thm:Dis_Representation}) and wish to determine a solution to a given~\eqref{eq.MultiObj_ch2} that has $p$ objectives, encoded as rows $1,\dots,p$ in the matrix $D\in\Re^{p\times n}$. It is well known in multiobjective optimization (and straightforward to prove) that adding an objective of the form $\lambda^{\top} D$ to the~\eqref{eq.MultiObj_ch2} does not change the set of efficient solutions, provided $\lambda\geq 0$, and nor does removing any duplicate objective. We may thus consider any restricted value problem with (i) $c^0=D^{k:k}$, the $k$th objective, and $C^{1:\ell}$ consisting of $D$ with the $k$th row deleted, so $\ell=p-1$, for any $k=1,\dots,p$, or (ii) $c^0=\lambda^{\top} D$ and $C^{1:\ell}=D$, so $\ell=p$, for any $\lambda\geq 0$. In any case of (i) or (ii), the set of efficient solutions of $\vinf\{Dx\;:\;x\in\XMO\}$ is precisely the set of efficient solutions of $\vinf\{Cx\;:\;x\in\XMO\}$. We may thus determine the EF of the former by finding $\Smin$ for the \VF{} associated with the latter. We are assured, by the first part of Theorem~\ref{thm:Relationship}, that if $\gamma\in\F$ where $(x'_I,x'_C)$ is an efficient solution with $\gamma=C_Ix'_I+C_Cx'_C$, then there exists $\zeta\in\cal C$ with $\zeta=\gamma^{1:\ell}$ and $z(\zeta)=\gamma_0$. Now Theorem~\ref{thm:Dis_Representation} ensures that for such $\zeta$, there must exist $(x_I,x_C)\in\XMO$ with  $x_I\in\Smin$ so that $\gamma=C_Ix_I+C_Cx_C$. Thus integer parts in $\Smin$ are sufficient to describe the EF, $\F$.   
	
	\section{Finite Algorithm for Construction}
	\label{sec:FiniteAlgorithm}
	
	In this section, we present an algorithm for constructing a discrete representation of both the \VF{} and the EF, as described in the previous section. 
 
	The purpose of the proposed algorithm is to construct a subset of $\S_I$, the elements of which constitute a description of the \VF{}, as codified in Theorem~\ref{thm:Dis_Representation}. The proposed algorithm is a generalized cutting-plane method that iteratively improves an upper approximation of the \VF{} until it converges to the true function. The ``cuts'' in this context refer to the convex bounding functions described in~\eqref{eq.RLPVF_ch2}. At iteration $k$, The upper approximation is the function $\bar{z}^k$, given by
	\begin{equation}
		\bar{z}^k (\zeta) = \min\left\{\min_{x_I \in S^{k}} \bar{z}(\zeta;\ x_I),\ U \right\}  \quad \forall \zeta \in \C,\label{eq.upper_approx}
	\end{equation}
	where $\S^{k}$ is the set of points in $\S_I$ identified so far and $U\in \Re$ is an upper bound on the VF, i.e. $z(\zeta)\leq U$ for all $\zeta\in\C$. The selection of the $U$ will be addressed in Section~\ref{sec:alginitterm}.
	
	The algorithm proceeds by identifying in iteration $k$ the point $\zeta^* \in \C$ at which the difference between the approximate value $\bar{z}^{k}(\zeta^*)$ and the true value $z(\zeta^*)$ is maximized. This point is given by $$\zeta^*=C_I^{1:\ell} x_I^{k+1} + C^{1:\ell}_C x_C^{k+1},$$
	where
	\begin{equation}\label{eq:defnextiterate}
		(x_I^{k+1}, x_C^{k+1}) \in \argmax\limits_{(x_I, x_C) \in \XMO} \left(\bar{z}^k(C_I^{1:\ell} x_I + C^{1:\ell}_C x_C) - (c^0_I x_I + c^0_C x_C)\right).
	\end{equation}
	Here $\bar{z}^k(C_I^{1:\ell} x_I^{k+1} + C^{1:\ell}_C x_C^{k+1}) = \bar{z}^k(\zeta^*)$ is the value of the upper approximation at $\zeta^*$ and---as we shall argue later---the true value at $\zeta^*$ is $z(\zeta^*)=c^0_I x_I^{k+1} + c^0_C x_C^{k+1}$. In the algorithm, this difference is denoted as	\begin{equation*}
		\begin{array}{ll@{}ll}
			\theta^{k+1} &= \bar{z}^k(C_I^{1:\ell} x_I^{k+1} + C^{1:\ell}_C x_C^{k+1}) - (c^0_I x_I^{k+1} + c^0_C x_C^{k+1})\\
			&= \max\limits_{(x_I, x_C) \in \XMO} \left(\bar{z}^k(C_I^{1:\ell} x_I + C^{1:\ell}_C x_C) - (c^0_I x_I + c^0_C x_C)\right).
	\end{array}\end{equation*}
	Note that $(x_I^{k+1}, x_C^{k+1})$ found via~\eqref{eq:defnextiterate} may be only {\em weakly} efficient, whereas only (strongly) efficient solutions are required to describe the \VF{}. Thus, before adding the integer part of $(x_I^{k+1}, x_C^{k+1})$ to $\S^{k}$, we first convert it to an efficient solution via a process common in multiobjective optimization: we replace $(x_I^{k+1}, x_C^{k+1})$ by an element of
	\begin{equation}\label{eq:conversion_to_NDP}
		\argmin_{(x_I, x_C) \in \XMO} \left\{\mathbf{1}^{\top} (C_I x_I + C_C x_C)\;:\;  C_I x_I + C_C x_C \leq C_I x_I^{k+1} + C_C x_C^{k+1}\right\}.
	\end{equation}
 It is easily proved that any optimal solution to~\eqref{eq:conversion_to_NDP} must also solve~\eqref{eq:defnextiterate}. (See Lemma~\ref{lem:conversionstillopt} in Appendix~\ref{appendix-not-wasteful}.) In the case of multiple optimal solutions for~\eqref{eq:defnextiterate}, solving~\eqref{eq:conversion_to_NDP} ensures an efficient solution for the multiobjective problem is selected. 
	Provided that the maximum difference between the upper approximation to the VF and the true VF is strictly positive, indicated by $\theta^{k+1}>0$, this yields a new stability region associated with a new member of $\S_I$, which we add to obtain $\S^{k+1}$. The algorithm is designed to iterate until the approximation is exact, detected by reaching $\theta^{k}=0$. 
 A high-level overview of the algorithm is provided below.
	\begin{algorithm}[htbp]
 \setcounter{AlgoLine}{0}
		\renewcommand\thealgorithm{RVF Algorithm}
		\SetAlFnt{\small}
		\SetAlCapFnt{\small}
		\SetAlCapNameFnt{\small}
		\small
		\SetAlgoLined
		\KwIn{$\XMO$, $C \in \Q^{(\ell+1) \times n}$, $U$ an upper bound on $z(\zeta)$ over $\zeta\in\C$.}
		\KwOut{$\S^k$ such that $z(\zeta) = \bar{z}^k(\zeta) = \min\limits_{x_I \in \S^k} \bar{z}(\zeta;\ x_I), \quad \forall \zeta \in \C.$}
		Initialize $\bar{z}^0(\zeta) = U$ for all $\zeta \in \Re^{\ell}$, $k = 0$, $\S^0 = \emptyset$, $\theta^0 = +\infty$.
		
		\While{$\theta^k > 0$}{
			Determine $(x_I^{k+1}, x_C^{k+1}) \in \argmax\limits_{(x_I, x_C) \in \XMO} \left(\bar{z}^k(C_I^{1:\ell} x_I + C^{1:\ell}_C x_C) - (c^0_I x_I + c^0_C x_C)\right)$. \label{step-3_ch2}
			
			Convert $(x_I^{k+1}, x_C^{k+1})$ to an efficient solution (for example, by using optimization problem~\eqref{eq:conversion_to_NDP}).
			
			Set $\S^{k+1} \leftarrow \S^{k}\cup \{x^{k+1}_I\}$.
			
			Set $\theta^{k+1} = \bar{z}^k(C_I^{1:\ell} x_I^{k+1} + C^{1:\ell}_C x_C^{k+1}) - (c^0_I x_I^{k+1} + c^0_C x_C^{k+1})$.
			
			$\bar{z}^{k+1}(\zeta) = \min \left\{\bar{z}^{k}(\zeta),\ \bar{z}(\zeta;\ x^{k+1}_I)\right\}$ for all $\zeta \in \C$.
			
			$k \leftarrow k+1$.
		}
		\caption{: Algorithm for constructing the \VF{} and the associated EF} \label{RVFAlgo_ch2}
	\end{algorithm}

	A few important details regarding the RVF Algorithm are worth noting. The RVF Algorithm generates a sequence of points in $\XMO$, but only stores the integer parts of these points in the set $\S^k$. 
 The upper approximation to the \VF{} at each iteration, denoted by $\bar{z}^k$, is initialized and then updated inductively at each iteration $k$ according to~\eqref{eq.upper_approx}. It is important to note that the upper approximation $\bar{z}^k$ is not explicitly constructed: the expression~\eqref{eq.upper_approx} is a {\em description} of the upper approximating function; it is the set $\S^k$ that is constructed by the algorithm. Thus, in the first step of each iteration, we need a method for optimization of the upper approximating function. We defer the explanation of this to Section~\ref{sec:MINLPderivation} below. Here we first illustrate the steps of the algorithm on Example~\ref{ex.MILPVF} and then establish its correctness.
	
	\begin{example}
		\textbf{Illustrative Example for the RVF Algorithm:}
		The steps of the~\ref{RVFAlgo_ch2} as applied to Example~\ref{ex.MILPVF} are depicted graphically in Figure~\ref{Fig:algoInstance} below. In the first iteration, the RVF Algorithm identifies the point ($\zeta, z(\zeta)) = (4\frac{4}{9}, 0)$, with stability region corresponding to $(x_1, x_2)=(0, 0)$, and updates the upper approximation to the blue convex function. In the subsequent iteration, the algorithm searches for the point with the largest difference between the current upper approximation and the RVF (the red piecewise linear function). The point of largest difference actually occurs at the far left, outside of the region shown in the figure, at $\zeta=-55.5$, but we illustrate the gap at $\zeta=-15$ (refer to Appendix~\ref{appendix-whole_EF} to view the full extent of the LP EFs). This point lies within the stability region corresponding to $(x_1, x_2)=(1, 0)$. The RVF Algorithm then updates the upper approximation again and, in the subsequent iteration, finds the point with the most difference between the upper approximation and the RVF, which occurs at $\zeta=-12\frac{1}{6}$. The RVF Algorithm continues in this manner, next finding $\zeta=-10$, until, at the next iteration, there is no such point. Thus the RVF Algorithm terminates in iteration 4 with $\theta = 0$ when the upper approximation and the VF are the same. 
		\begin{figure}[tb]
			\centering
			\begin{subfigure}{0.43\textwidth}
				\includegraphics[width=1\linewidth]{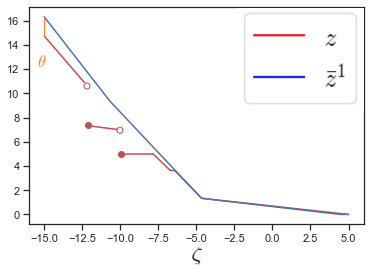}
			\end{subfigure}\hfil
			\begin{subfigure}{0.43\textwidth}
				\includegraphics[width=1\linewidth]{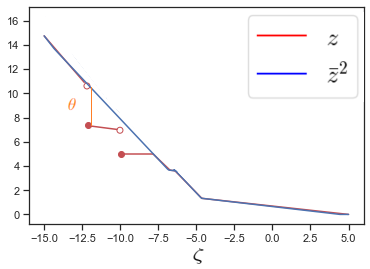}
			\end{subfigure}
			\medskip
			\begin{subfigure}{0.43\textwidth}
				\includegraphics[width=1\linewidth]{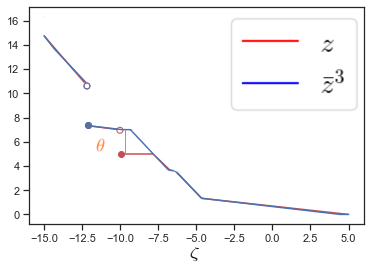}
			\end{subfigure}\hfil
			\begin{subfigure}{0.43\textwidth}
				\includegraphics[width=1\linewidth]{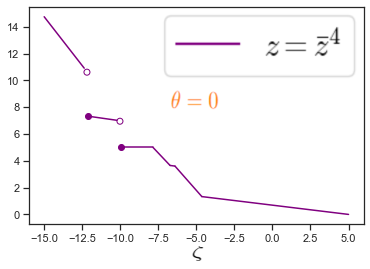}
			\end{subfigure}\hfil
			\caption{RVF and upper approximation in iterations of the RVF Algorithm with Example~\ref{ex.MILPVF}. The RVF $z$ is illustrated on the range $\zeta \in [-15, 5]$. Note that the first cut occurs at
$\zeta = -55.5$, which is outside of this range. Please refer to Appendix~\ref{appendix-whole_EF} to view the full EF or RVF.}
			\label{Fig:algoInstance}
		\end{figure}
	\end{example}

	\subsection{Correctness of the RVF Algorithm}
	\label{sec:RVFAlgCorrect}
	
	The proof of the correctness of the~\ref{RVFAlgo_ch2} relies on the fact that $\bar{z}^k$ is a nonincreasing function for the same reason that the \VF{} is nonincreasing, namely that it is the minimum of a finite set of nonincreasing bounding functions by~\eqref{eq.upper_approx}. Additionally, in iteration $k$ of the algorithm, $\theta^{k+1} = \max_{\zeta \in \C} \bar{z}^k(\zeta)-z(\zeta)$ (see Lemma~\ref{lem:thetaismaxoverzeta} in Appendix~\ref{appendix:correctness}). Furthermore, any optimal solution $(x_I^{k+1}, x_C^{k+1})$ to the optimization problem in Equation \eqref{eq:defnextiterate} having optimal value $\theta^{k+1} > 0$ has the property that $x_I^{k+1} \not\in {\cal S}^k$ (see Lemma ~\ref{lem:Subprobsolnnewimage} in Appendix~\ref{appendix:correctness}.)
	
	Now, with Theorem~\ref{Correctness}, we show the correctness of the RVF Algorithm by proving that it terminates finitely and returns the correct VF.
	
	\btheorem
	\label{Correctness}(Correctness of the RVF Algorithm)
	At termination, we have that $z(\zeta) = \bar{z}^k(\zeta)$ for all $\zeta \in \C$, so that $\S^k$  describes both the VF and the EF. Furthermore, the RVF Algorithm terminates in finitely many iterations under the assumption that $\XMO$ is bounded.  
	\etheorem
	\begin{proof}
		When $\XMO$ is bounded, $\S_I$ is finite. Therefore, by Lemma~\ref{lem:Subprobsolnnewimage}, which states that each iteration of the algorithm produces a new element of $\S_I$, the number of iterations must be finite. To show that $z = \bar{z}^k$ at termination,
		assume not for the sake of contradiction. Then there must exist $\zeta \in \Re^{\ell}$ such that $z(\zeta) < \bar{z}^k(\zeta)$. But, by Lemma~\ref{lem:thetaismaxoverzeta}, this is a contradiction to the assumption that $\theta^k = 0$ at termination. This completes the proof.
	\end{proof}
	
	Proposition~\ref{prop:newitnonemptystab} in Appendix~\ref{appendix-not-wasteful} shows that the algorithm has one seemingly advantageous property that we briefly describe here. As the description is being constructed, each newly added integer part is guaranteed to be associated with a nonempty stability region. In fact, it is not difficult to see that the added region must be \emph{non-redundant} at the time it is added. In other words, we must have that $\bar{z}^k \not= \bar{z}^{k+1}$ at iteration $k$. Equivalently, there exists $\zeta \in \C$ such that $\bar{z}^{k+1}(\zeta) < \bar{z}^k(\zeta)$, namely $\zeta^{k+1}$ from the proof of Lemma~\ref{lem:thetaismaxoverzeta}. Unfortunately, this property does not translate into a guarantee that the stability region associated with $x^k$ is non-redundant in the end, so the algorithm is not guaranteed to produce a minimal description. It would be possible to postprocess the description to make it minimal if this was important for a particular application. 
	
	\subsection{Solving the Subproblem}\label{sec:MINLPderivation}
	
	The~\ref{RVFAlgo_ch2} was presented in the previous section at a high level of abstraction to simplify the exposition. In the next two sections, we provide further details on how the algorithm can actually be implemented in practice. In this section, we start by clarifying how the subproblem that needs to be solved in each iteration can be solved in practice by formulating it as a standard mathematical optimization problem.
	
	The optimization problem in~\eqref{eq:defnextiterate}, solved in Step~\ref{step-3_ch2} of the algorithm, can be formulated as a mixed integer nonlinear optimization problem (MINLP) as follows. First, we model it using an auxiliary variable $\theta$ to move $\bar{z}^k$ into the constraints:
	\begin{equation}
		\begin{array}{ll@{}ll}
			\theta^{k+1}=\max  &\theta\\
			\text{subject to}&\theta \leq \bar{z}^k(C_I^{1:\ell} x_I + C_C^{1:\ell} x_C) - (c^0_I x_I + c^0_C x_C)\\
			& (x_I, x_C) \in \XMO\\
			& \theta \in \Re. 
		\end{array} \label{eq:SubProbModela}
	\end{equation}
	Next, we use~\eqref{eq.upper_approx} and~\eqref{eq.CR_ch2} to expand~\eqref{eq:SubProbModela} to the $k+1$ constraints
	\begin{equation}\label{eq:SubProbModelbU}
		\theta + c^0_Ix_I + c^0_C x_C \leq U,
	\end{equation}
	and
			\begin{equation}\label{eq:SubProbModelbzLP}
		\theta + c^0_I x_I + c^0_C x_C \leq c^0_I x_I^i + \zLP(C_I^{1:\ell} x_I + C_C^{1:\ell} x_C - C^{1:\ell}_I x_I^i;\ b - A_I x_I^i), \quad i = 1, \hdots, k.
	\end{equation}

	(Recall $\S^k=\{x^1_I,\dots,x^k_I\}$.)
	We can model the term involving $\zLP$ for each $i$ by using its LP dual problem, which via~\eqref{eq.D_RLP} is given by
	\begin{equation}
		\max\limits_{(u^i, v^i) \in \PD} (C_I^{1:\ell} x_I + C_C^{1:\ell} x_C - C^{1:\ell}_I x_I^i)^{\top} u^i + (b - A_I x_I^i)^{\top} v^i.
		\label{z_LP_algo}
	\end{equation}
	If $(x_I,x_C)\in\XMO$ satisfies~\eqref{eq:SubProbModelbzLP} for some $i$, then by strong LP duality, there must exist $(u^i, v^i)\in\PD$ (optimal for the dual of the LP after fixing the values of the integer variables to $x^i_I$) such that 
	\begin{eqnarray*}
		\theta + c^0_I x_I + c^0_C x_C & \leq & c^0_I x_I^i + \zLP(C_I^{1:\ell} x_I + C_C^{1:\ell} x_C - C^{1:\ell}_I x_I^i;\ b - A_I x_I^i) \\
		& = & c^0_I x_I^i + (C_I^{1:\ell} x_I + C_C^{1:\ell} x_C - C^{1:\ell}_I x_I^i)^{\top} u^i + (b - A_I x_I^i)^{\top} v^i.
	\end{eqnarray*}
	Conversely, if there exists some $(u^i, v^i)\in\PD$ such that $(x_I,x_C)\in\XMO$ satisfies
	$$
	\theta + c^0_I x_I + c^0_C x_C \leq c^0_I x_I^i + (C_I^{1:\ell} x_I + C_C^{1:\ell} x_C - C^{1:\ell}_I x_I^i)^{\top} u^i + (b - A_I x_I^i)^{\top} v^i,
	$$
	then by weak duality, it must be that $(x_I,x_C)$ satisfies~\eqref{eq:SubProbModelbzLP} for this $i$. This shows that the following explicit MINLP is a valid formulation of the optimization problem in~\eqref{eq:defnextiterate}:
	\begin{align}
		\theta^{k+1}=\max \quad &\theta& \label{eq:final_obj}\\
		\text{subject to} \quad &\theta + c^0_I x_I + c^0_C x_C \leq c^0_I x_I^i +& \nonumber\\
		&(C_I^{1:\ell} x_I + C_C^{1:\ell} x_C - C^{1:\ell}_I x_I^i)^{\top} u^i +
		(b - A_I x_I^i)^{\top} v^i, \quad &i = 1, \hdots, k &\label{eq:theta_const}\\
		&\theta + c^0_I x_I + c^0_C x_C \leq U  &\label{eq:bounding_const}\\
		& (u^i, v^i)\in \PD, \quad &i = 1, \hdots, k & \label{eq:dual_var_const}\\
		& (x_I, x_C) \in \XMO& \label{eq:primal_var_const}\\
		&\theta \in \Re.&\label{eq:theta_var_const}
	\end{align}
 
Note that the constraint~\eqref{eq:theta_const} makes the problem nonlinear: it has bilinear terms $(C_I^{1:\ell} x_I)^{\top} u^i$ and $(C_C^{1:\ell} x_C)^{\top} u^i$ linking the $u^i$ and $(x_I,x_C)$ variables for each $i$. However, the problem~\eqref{eq:final_obj}--\eqref{eq:theta_var_const}, which we refer to as the {\em MINLP subproblem}, can be solved either with an off-the-shelf nonconvex quadratic solver or possibly with a customized algorithm.  
 
	\subsection{Initialization and Termination}\label{sec:alginitterm}
	
	The most important aspect of algorithm initialization is determining an initial upper bound $U$. 
 One possible approach is to solve the LP relaxation of the MILP $\max\{c^0_Ix_I + c^0_Cx_C\;:\;(x_I,x_C)\in\XMO\}$. This would provide a valid initial upper bound for the optimization problem.
	
	It is important to note that in the final iteration, when $\theta^{k+1}$ takes value 0, there is no need to convert the solution $(x^{k+1}_I,x^{k+1}_C)$ into an NDP and add it to $\S^k$ {\em unless} its objective value $c^0_Ix^{k+1}_I + c^0_Cx^{k+1}_C$ happens to coincide with the initial upper bound, $U$. 
	
	Finally, observe that if the~\ref{RVFAlgo_ch2} is terminated early, while $\theta^k>0$, then the value of the final optimization problem solved in Step 2 provides a natural performance guarantee on the quality of the current upper approximation to the VF: it measures the maximum error between the true VF, $z$, and the upper approximation, $\bar{z}^k$, over any point in its domain.

	\subsection{Publicly Available Implementation}
	
	We provide a Python package, implemented using the RVF Algorithm, for enumerating all integer parts required to construct the NDPs for instances of multiobjective integer and mixed integer programs with an arbitrary number of objective functions. The MINLP subproblems encountered at each iteration are solved using Couenne~\citep{belotti2009couenne}. This package is available at \url{https://github.com/SamiraFallah/RestrictedValueFunction}. 

 
	\section{Conclusions}
	\label{sec:conclusions}
	In this study, we discussed the relationship between the restricted value function, \VF{}, and the EF of a multiobjective integer linear program~\eqref{eq.MultiObj_ch2}. We demonstrated that the EF lies on a subset of the boundary of the epigraph of the \VF{}. In so doing, we highlight an important relationship that connects two parts of the literature that had been considered distinct. We also demonstrated that the \VF{} is the minimum of polyhedral functions associated with \LPVF{} and discussed the structure of the \VF{}, including its continuity and differentiability. Finally, we showed that under the assumption that the $\XMO$ is bounded, there exists a finite description of both \VF{} and the EF in which each element of the description corresponds to a stability region. 
	
 In this context, we have introduced the \VF{} algorithm with the aim of identifying all efficient integer solutions essential for constructing the RVF or EF in the domain of a multiobjective optimization problem. 
 Our proposed algorithm offers an alternative that operates on quite different principles from existing algorithms for constructing the frontier. The algorithm is somewhat unique in the MO-MILP literature in its degree of generality---it can construct a discrete representation of the EF for arbitrarily many objectives and in the presence of continuous variables, as well as provide a performance guarantee if terminated early. We also provide a Python package that demonstrates the practical effectiveness of the proposed algorithm.

 The basic algorithm we introduced can be improved in several ways and work on enhancing its practical efficiency is ongoing. For example, it is clear that solving the MINLP subproblem in each iteration from a cold start is wasteful when the problem is only slightly modified from one already solved in the previous iteration. We hope that the ideas presented here will open up new avenues of inquiry with respect to algorithmic development in multiobjective optimization, as well as broadening the applications of such ideas into other areas not seen as related until now. 

\section*{Declaration of Competing Interest}
The authors declare that they do not have any identifiable conflicting financial interests or personal relationships that could have potentially influenced the findings presented in this paper.
 
	\bibliographystyle{plainnat} 
	\bibliography{refs} 
	
	\pagebreak
\begin{appendices}
 \section{An Instance of RLPVF}
	\label{RLPVF}
\bexample\label{ex.non-differ}
	In Figure~\ref{Fig:RLPVF}, we plot an instance of~\eqref{eq.RLPVF_ch2} obtained from Example~\ref{ex.MILPVF} by setting $(x_1, x_2)=(0, 0)$. In the plot, the affine 
	functions associated with each extreme point of $\PD$ are shown for a part of the finite domain. For this example, 
	the dual problem~\eqref{eq.D_RLPVF} can be expressed as 
\begin{align*}\label{dual_ex} \tag{Ex1-D}
		\zLP(\zeta) =  \max \;\; &\zeta u + 4v_1 + 5v_2 + 5v_3\\
		 \emph{s.t.}\;\; ~  &10u + 9v_1 \leq 0,~ -8u + 3v_1 + v_2 \leq 7,~u + 2v_1 \leq 10,~-7u + 6v_1 \leq 2\\
		&6u -10v_1  + v_3 \leq 10, ~u \leq 0, ~v_2 \leq 0, ~v_3 \leq 0,\\
	\end{align*}
	and we have that the set of extreme points, $\Extr$, and the set of extreme rays, $\R$, of $\PD$ are
	\begin{align*}
		\Extr = \{&(-0.15,  0.16, 0, 0), (0, 0, 0, 0), (-2.35, -2.41, -4.59, 0), \\
		& (-1.33, -1.22, 0, 0), (-1.61, -1.97, 0, 0),(0, -1, 0, 0)\}, \\
		\R = \{&(0, 0, 0, -1), (0, -1, 0, -10), (-1, -2.66, 0, -20.66), \\
		& (-1, -1.166, -4.5, -5.66), (0, 0, -1, 0)\}.
	\end{align*}
	As expected, the function defined by~\eqref{eq.RLPVF_ch2} is a polyhedral function with the nondifferentiable points being those at which the optimal basis changes.   
	\begin{figure}[htbp]
		\centering
		\includegraphics[width=0.80\linewidth]{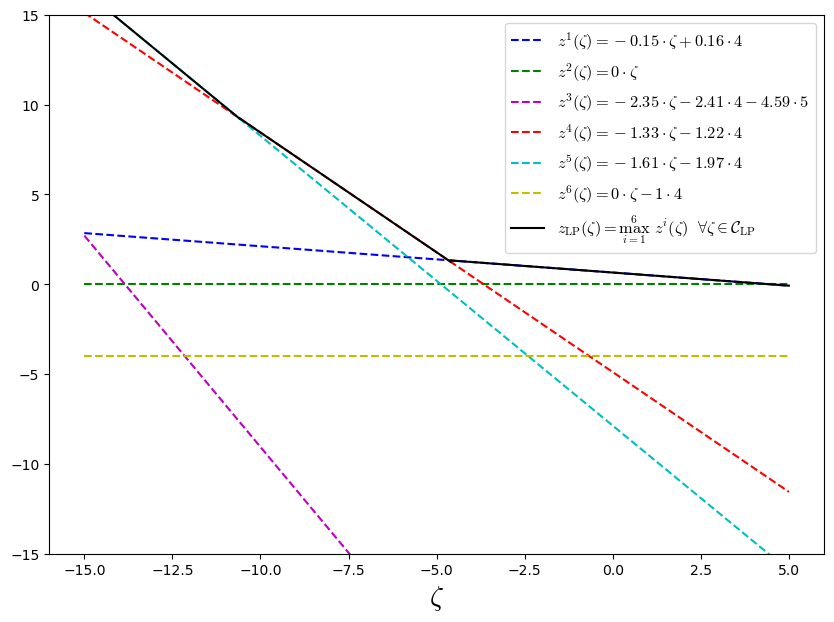}
		\caption{The RLPVF associated with Example~\ref{ex.MILPVF} when $(x_1, x_2)=(0, 0)$} \label{Fig:RLPVF}
	\end{figure} 
	Not all affine pieces are needed to describe the function and a minimal description is as follows. The finite domain is $\CLP = [-55.464, +\infty)$ and for $\zeta \in \CLP$, we have
\begin{align*} \label{dual_ex_bfs}
		\zLP(\zeta) = \max(&0, -2.35\zeta - 2.41\cdot 4 -4.59\cdot 5, -1.61\zeta -1.97\cdot 4, -1.33\zeta -1.22\cdot 4, -0.15\zeta + 0.16\cdot 4).
	\end{align*}
	Observe that the gradients of the RLPVF are precisely the optimal dual solutions corresponding to the parametric constraints. At points of nondifferentiability (the breakpoints), the directional derivatives in direction $d$ (in this case, we have $d \in \{1, -1\}$) of the RLPVF are given by $d^\top u$, where $u$ is one of the alternative optimal dual solutions corresponding to the parametric constraint.
	\eexample

		\section{The LP Frontiers of Example~\ref{ex.MILPVF}}
		\label{appendix-whole_EF}
		The entire EF for Example~\ref{ex.MILPVF} is shown below in Figure~\ref{Fig:completeFrontier}. The bounding functions associated with each integer vector of Example~\ref{ex.MILPVF} are illustrated in Figure~\ref{fig:fullEx1}. To highlight the detail in the bottom right parts of the frontiers more clearly, an enlarged view is presented in Figure~\ref{fig:truncleftEx1}.
		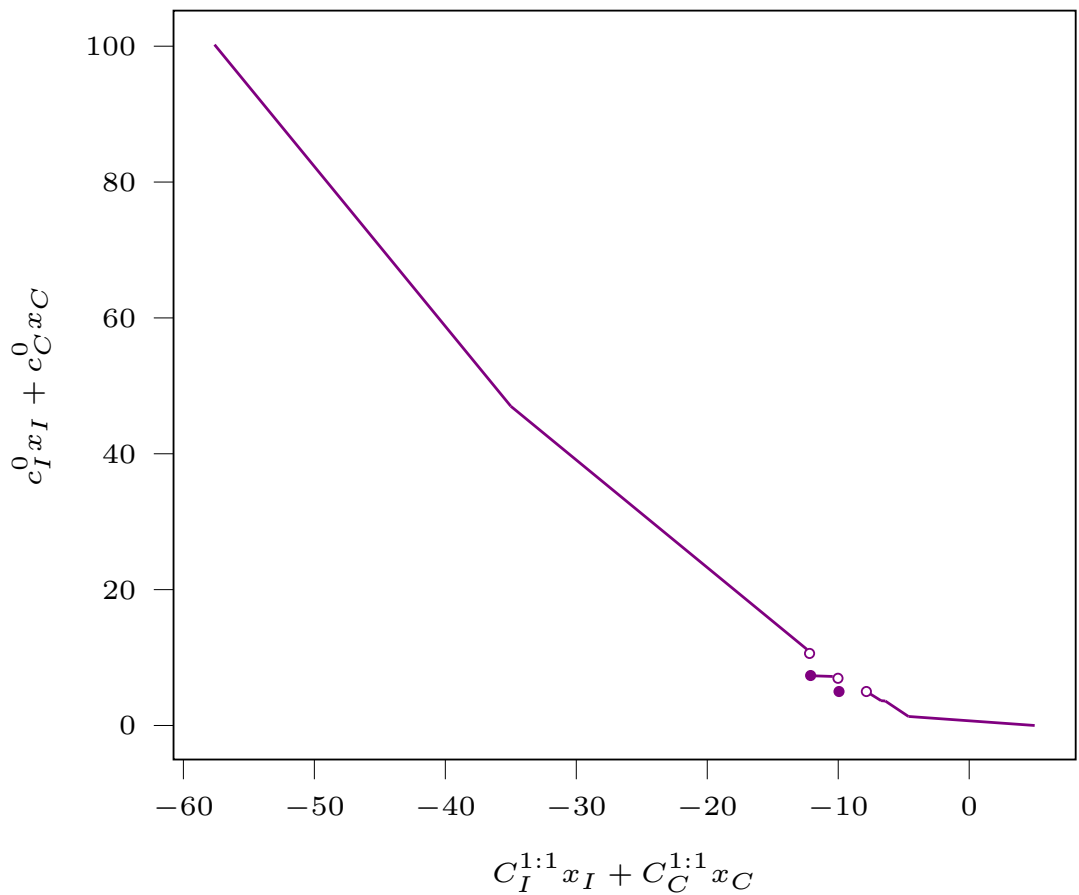
\begin{figure}[h]
			\centering
			\begin{tikzpicture}[scale=1.75]
				\definecolor{darkgray176}{RGB}{176,176,176}
				\definecolor{purple}{RGB}{128,0,128}
				\begin{axis}[
					tick align=outside,
					tick pos=left,
					x grid style={darkgray176},
					xlabel={\(\displaystyle C_I^{1:1} x_I + C_C^{1:1} x_C\)},
					xmin=-60.7523046092184, xmax=8.1310621242485,
					xtick style={color=black},
					y grid style={darkgray176},
					ylabel={\(\displaystyle c^{0}_I x_I + c^{0}_{C} x_C\)},
					ymin=-5.01132266, ymax=105.23777586,
					ytick style={color=black},
					yticklabel style = {font=\tiny},
					xticklabel style = {font=\tiny},
					ylabel style = {font=\tiny},
					xlabel style = {font=\tiny}
					]
					\addplot [draw=purple, fill=purple, mark=*, only marks, mark size=1pt]
					table{
						x  y
						-12.1084337349398 7.35481144
						-9.93975903614458 5
					};
					\addplot [draw=purple, mark=o, only marks, mark size=1pt]
					table{
						x  y
						-12.1887550200803 10.60722886
						-10.0200803212851 6.95
						-7.85140562248996 5
					};
					\addplot [semithick, purple]
					table {
						-57.6212424849699 100.2264532
						-34.9999720155731 46.9999541320932
					};
					\addplot [semithick, purple]
					table {
						-34.9999720155731 46.9999541320932
						-12.3887550200803 11.20722886
					};
					\addplot [semithick, purple]
					table {
						-12.1084337349398 7.32481144
						-10.4200803212851 7.2191987462
					};
					\addplot [semithick, purple]
					table {
						-7.55140562248996 4.64
						-6.72690763052209 3.65109218
					};
					\addplot [semithick, purple]
					table {
						-6.72690763052209 3.65109218
						-6.40562248995984 3.60407484
					};
					\addplot [semithick, purple]
					table {
						-6.40562248995984 3.60407484
						-4.63855421686747 1.32921932
					};
					\addplot [semithick, purple]
					table {
						-4.63855421686747 1.32921932
						5 0
					};
				\end{axis}
				
			\end{tikzpicture}
			\caption{The EF associated with Example~\ref{ex.MILPVF}} \label{Fig:completeFrontier}
		\end{figure}

		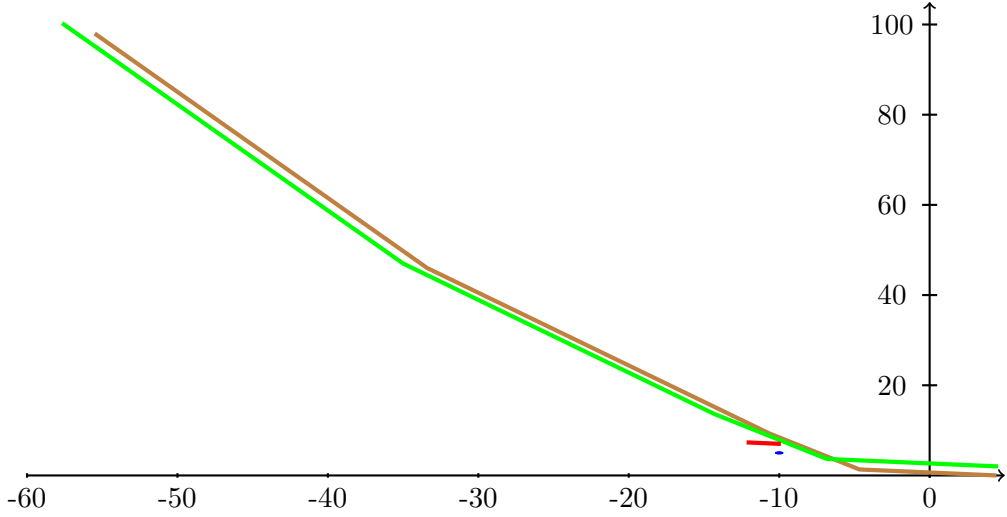
\begin{figure}[h]
			\centering
			\begin{tikzpicture}[yscale=0.06,xscale=0.2]
				\def \w {0.5} 
				\def \o {2.5}  
				\def \c {10} 
				\draw[->, Black, thick] (0,-\w) -- (0,105);
				\foreach \i in {1,...,5}
				{
					\draw[Black, thick] (-\w,2*\c*\i) -- (\w,2*\c*\i);
					\node at (-\o,2*\c*\i) {\the\numexpr 2*\c*\i};
				};
				\draw[->, Black, thick] (-60,0) -- (5,0);
				\foreach \i in {0,...,6}
				{
					\draw[Black, thick] (-\c*\i,-\w) -- (-\c*\i,\w);
					\node at (-\c*\i,-\o-\o) {\the\numexpr -\c*\i};
				};
				\draw[Sepia, ultra thick] (-55.5,98) -- (-33.4,46) --  (-10.66667,9.33333) -- (-4.66667,1.33333) -- (4.444444,0);
				\draw[ForestGreen, ultra thick] (-57.6667,100.3) -- (-35,47) -- (-14.3333,13.66667) -- (-6.83333,3.66667) -- (4.5555556,2);
				\draw[Red, ultra thick] (-9.88889,7) -- (-12.1667,7.33333);
				\draw [Blue, fill=blue] (-10,5) circle [radius=0.25];;
			\end{tikzpicture}
			\caption{The LP frontiers. Brown is for $(x_1,x_2)=(0,0)$. Green is for $(x_1,x_2)=(1,0)$. Red is for $(x_1,x_2)=(1,1)$. Blue is for $(x_1,x_2)=(0,1)$. Note the axes are unequally scaled.}\label{fig:fullEx1}
		\end{figure}
		
		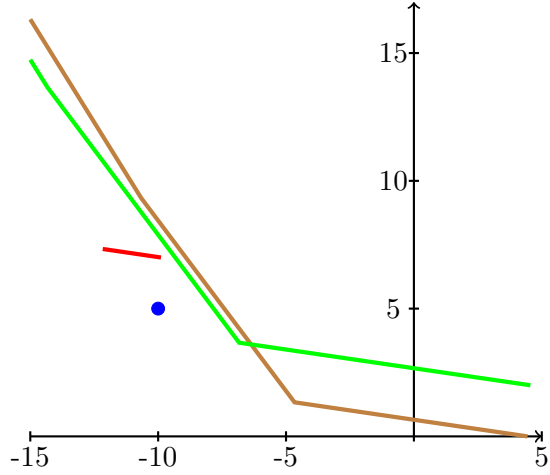
\begin{figure}[h]
			\centering
			\begin{tikzpicture}[scale=0.34]
				\draw[->, Black, thick] (0,-0.2) -- (0,17);
				\draw[Black, thick] (-0.2,10) -- (0.2,10);
				\node at (-0.8,10) {10};
				\draw[Black, thick] (-0.2,5) -- (0.2,5);
				\node at (-0.8,5) {5};
				\draw[Black, thick] (-0.2,15) -- (0.2,15);
				\node at (-0.8,15) {15};
				\draw[->, Black, thick] (-15,0) -- (5,0);
				\draw[Black, thick] (-15,-0.2) -- (-15,0.2);
				\node at (-15,-0.8) {-15};
				\draw[Black, thick] (-10,-0.2) -- (-10,0.2);
				\node at (-10,-0.8) {-10};
				\draw[Black, thick] (-5,-0.2) -- (-5,0.2);
				\node at (-5,-0.8) {-5};
				\draw[Black, thick] (5,-0.2) -- (5,0.2);
				\node at (5,-0.8) {5};
				\draw[Sepia, ultra thick] (-15,16.32258) -- (-10.66667,9.33333) -- (-4.66667,1.33333) -- (4.444444,0);
				\draw[ForestGreen, ultra thick] (-15,14.74194) -- (-14.3333,13.66667) -- (-6.83333,3.66667) -- (4.5555556,2);
				\draw[Red, ultra thick] (-9.88889,7) -- (-12.1667,7.33333);
				\draw [Blue, fill=blue] (-10,5) circle [radius=0.25];;
			\end{tikzpicture}
			\caption{The (truncated at the left) LP frontiers. Brown is for $(x_1,x_2)=(0,0)$. Green is for $(x_1,x_2)=(1,0)$. Red is for $(x_1,x_2)=(1,1)$. Blue is for $(x_1,x_2)=(0,1)$}\label{fig:truncleftEx1}
		\end{figure}

  \section{Proofs Regarding the Structure of RVF}
	\label{appendix-proofs-sec3}
Here, we present the detailed proofs of results from Section~\ref{sec:MILP_VF}, along with the statements of and proofs for supporting lemmas and propositions.

 \bproposition
	\label{prop:differentiable_VF}
	For $\zeta \in \CLP$, we have that
	$$
	\nabla_d \zLP(\zeta) = 
	\begin{cases} 
		\max_{u\in\popt{\zeta}} u^{\top}d, & \textrm{if } d \not\in \delta^-(\zeta), \\
		+\infty, & \textrm{otherwise},
	\end{cases}
	$$
	for all $d \in \Re^{\ell}$.
	\eproposition

\begin{proof}
		We let $\zeta\in\CLP$ and $d \in \Re^{\ell}$ be given and consider $\nabla_d \zLP(\zeta)$. There are two cases. 
		\begin{enumerate}
			\item[(i)] If $d \in \delta^-(\zeta)$, then we have already seen that $\nabla_d \zLP(\zeta) = +\infty$, since $\zLP(\zeta + \epsilon d) = +\infty$ for all $\epsilon > 0$. Equivalently, we have that $\popt{\zeta + \epsilon d} = \emptyset$ and thus $\max_{u\in\popt{\zeta}} u^{\top}d = +\infty$, proving the result in this first case.
			
			\item[(ii)] If $d \not\in \delta^-(\zeta)$, by~\eqref{eq.D_RLPVF}, there exists $\epsilon > 0$ such that $\zLP(\zeta + \epsilon d) < +\infty$. As such, there exists $(\hat{u}, \hat{v})\in\Extr$ such that 
			\begin{equation}\label{eq:dir1piece}
				\zLP(\zeta+td) = (\zeta+td)^{\top} \hat{u} + \beta^{\top} \hat{v}, \qquad \forall t\in[0,\epsilon].
			\end{equation} 
			Thus $\nabla_d \zLP(\zeta) = \hat{u}^{\top}d$. Note that by taking $t=0$ we have that $\hat{u} \in\popt{\zeta}$. Now suppose, for the sake of contradiction, that $\overline{u}^{\top}d > \hat{u}^{\top}d$ for some $\overline{u}\in\popt{\zeta}$. Then by~\eqref{eq:dualoptfaceconvextr}, there must exist $(u, v) \in\Extr$ with $u\in\popt{\zeta}$ and $u^{\top}d > \hat{u}^{\top}d$. But then for any $t\in (0,\epsilon]$, we have
			\begin{equation*}
				\begin{aligned}
					(\zeta+td)^{\top} u+ \beta^{\top} v = \zLP(\zeta) + td^{\top} u
					> \zLP(\zeta) + td^{\top} \hat{u}
					&=\zeta^{\top}\hat{u}+\beta^{\top} \hat{v} + td^{\top} \hat{u}\\
					&=\zLP(\zeta+td),
				\end{aligned}
			\end{equation*}
			where the final equality follows from~\eqref{eq:dir1piece},
			which contradicts~\eqref{eq.D_RLPVF} at $\zeta+td$. The result follows.
	\end{enumerate}
\end{proof}

\begin{proof}[Proof of Prosition~\ref{prop:subdiffeqdualopt}]

The result follows from lemmas~\ref{lem:dualoptinsubdiff_restricted}, \ref{lem:epi_z_lp}, and \ref{lem:subdiffindualopt_restricted}, which together show containment and reverse containment of $\popt{\zeta}$ and $\partial \zLP(\zeta)$. Lemma~\ref{lem:epi_z_lp} is a known result needed to prove Lemma~\ref{lem:subdiffindualopt_restricted}, which we prove here again because part of the proof is referred to later.

\end{proof}

\begin{lemma} \label{lem:dualoptinsubdiff_restricted}
		For all $\zeta \in \CLP$, we have that
		$$
		\popt{\zeta} \subseteq \partial \zLP(\zeta).
		$$
	\end{lemma}

\begin{proof}
		Let $\hzeta \in \CLP$ and $\hu \in \popt{\hzeta}$ be given. We show that $\hu \in \partial \zLP(\hzeta)$. Let $\hat{v}$ be such that $(\hu,\hat{v})\in\PD$ and $\hzeta^{\top}\hu + \beta^{\top}\hat{v} = \zLP(\hzeta)$. Such $\hat{v}$ must exist by the definitions of $\popt{\hzeta}$. Then for any $\zeta \in \CLP$, we have
		\begin{equation}\label{eq:ghbeqzhb}
			\begin{aligned}
				(\zeta - \hzeta)^\top \hu + \zLP(\hzeta) & = \zeta^\top \hu - \hzeta ^\top \hu + \hzeta^\top \hu + \beta \hat{v}\\
                & = \zeta^\top \hu + \beta \hat{v} \\
                & \leq \zLP(\zeta).
		\end{aligned}\end{equation}
  It follows directly that $\hu \in \partial \zLP(\hzeta)$. Since $\hzeta$ was chosen arbitrarily, the result follows. 
\end{proof}

\begin{lemma}[\citep{ralphs2014value}] \label{lem:epi_z_lp}
	\begin{equation}
		\begin{aligned}
			\label{eq:epidual}
			\epi \zLP = \big\{(\zeta, w)\;:\; &\zeta^{\top} e + \beta^{\top} h \leq 0,\ \forall (e,h)\in \R \text{ and } \\
			&w \geq\zeta^{\top} u + \beta^{\top} v,\ \forall (u,v)\in\Extr \big\}.
		\end{aligned}
	\end{equation}
\end{lemma}
\begin{proof} 
 By Proposition~\ref{epiGraphConvex}, $\zLP$ is a polyhedral function, so its epigraph, given by
	$$
	\epi \zLP = \left\{(\zeta,w)\in\Re^{\ell}\times\Re\;:\; \zeta \in \CLP,\ w \geq \zLP(\zeta)\right\},
	$$
	is a polyhedron. By the definition of the subdifferential, for any $\hzeta \in \CLP$ and subgradient $g \in \partial \zLP(\hzeta)$, the hyperplane
	\begin{equation*}
		\label{eq:hyper}
		\left\{(\zeta, w) \in \Re^{\ell} \times \Re: (-g,1)^{\top} \left( \begin{array}{c} \zeta \\ w \end{array} \right) = \zLP(\hzeta) - g^{\top}\hzeta \right\},
	\end{equation*}
	is a hyperplane that supports epi $\zLP$ at $\zeta = \hzeta$ and the inequality
	\begin{equation}
		\label{eq:hypervalid}
		(-g,1)^{\top} \left( \begin{array}{c} \zeta \\ w \end{array} \right) \geq \zLP(\hzeta) - g^{\top}\hzeta,
	\end{equation}
	is satisfied by all $(\zeta, w)\in\epi\zLP$, with equality attained at $(\hzeta,\zLP(\hzeta))$. 
	
	The proof also makes use of the extreme points and extreme rays of the dual LP feasible set. Recall that $\Extr$ and $\R$ are the (finite) sets of extreme points and extreme rays of $\PD$, respectively. Recall that for $\zeta \in \CLP$, we have that 
	\begin{equation*}
		\label{eq:bfeasextr}
		\zeta^{\top} e + \beta^{\top} h \leq 0, \qquad\forall (e, h)\in \R,
	\end{equation*}
	and that by LP duality, we have
	\begin{equation*}
		\label{eq:dualLPvalextr}
		\zLP(\zeta) = \max \left\{\zeta^{\top} u + \beta^{\top} v \;:\;\ (u,v)\in\Extr\right\}.
	\end{equation*} \end{proof}

	\begin{lemma} \label{lem:subdiffindualopt_restricted}
		For all $\zeta \in \CLP$,
		$$
		\partial \zLP(\zeta) \subseteq \popt{\zeta}.
		$$
	\end{lemma}

\begin{proof}
		Let $\hzeta \in \CLP$ and $g\in \partial \zLP(\hzeta)$ be given. Then by Lemma~\ref{lem:epi_z_lp} (and \eqref{eq:hypervalid}), it must be that $(\hzeta, \zLP(\hzeta))$ is an optimal solution of the LP
		\begin{equation} \label{eq:extrLP}
			\begin{aligned}
				\left\{\begin{array}{lll}
					\min & (-g,1)^{\top} \left( \begin{array}{c} \zeta \\ w \end{array} \right) \\
					\text{s.t.} & (\zeta, w)\in \epi \zLP
				\end{array}\right.
				= \left\{\begin{array}{lll}
					\min_{\zeta, w} &-g^{\top} \zeta + w \\
					\text{s.t.} & \zeta^{\top} e + \beta^{\top} h \leq 0, & \forall (e, h)\in \R,\\
					& w \geq\zeta^{\top} u + \beta^{\top} v, & \forall (u, v)\in \Extr,
				\end{array} \right.
			\end{aligned}
		\end{equation}
		which has LP dual
		\begin{equation}\label{eq:extrLPdual}
			\begin{array}{ll}
				\max_{\lambda,\gamma} & \displaystyle\sum_{(u, v)\in\Extr} \lambda_{uv} \beta^{\top} v + \sum_{(e, h)\in\R} \gamma_{eh} \beta^{\top} h\\
				\text{s.t.} & \displaystyle\sum_{(u, v)\in\Extr} \lambda_{uv} u + \sum_{(e, h)\in\R} \gamma_{eh} e = g\\
				&  \displaystyle\sum_{(u,v)\in\Extr}\lambda_{uv} = 1\\
				& \lambda\in \Re_+^\Extr ,\ \  \gamma \in \Re_+^\R,
			\end{array} 
		\end{equation}
		where the $\gamma$ is the vector of dual variables associated with the first set of constraints, and $\lambda$ is the vector of dual variables associated with the second set of constraints. Recall that $(\hzeta,\zLP(\hzeta))$ is an optimal solution of~\eqref{eq:extrLP} and let $(\lambda^*,\gamma^*)$ be an optimal solution of~\eqref{eq:extrLPdual}. Then taking 
  \begin{equation}
      \begin{aligned}
      (u^*, v^*) = \sum_{(u, v)\in\Extr} \lambda^*_{uv} (u, v) \\
      (e^*, h^*) = \sum_{(e, h)\in\R} \lambda^*_{eh} (e, h),
      \end{aligned}
  \end{equation}
we have that $(u^* + e^*, v^* + h^*) \in \PD$ and also that $ \hzeta^\top (u^* + e^*) = \zLP(\hzeta) - \beta^\top (v^* + h^*)$. And finally, since the constraints state that $g = u^* + e^*$, we have that $g \in \popt{\hzeta}$. Since $\hzeta$ was arbitrary, the result follows. 
\end{proof}

	\begin{corollary}
 \label{cor-sec3}
		$\zLP(\zeta)$ is differentiable at $\zeta$ if and only if $\zeta\in\inter \CLP$ and the dual problem (\ref{eq.D_RLP}) has a unique optimal solution.
	\end{corollary}
	\begin{proof}
For $\zeta\in\inter \CLP$, $\popt{\zeta}$ consists of a singleton if and only if (\ref{eq.D_RLP}) has a unique optimum. By Proposition~\ref{prop:subdiffeqdualopt}, the subdifferential $\partial \zLP(\zeta)$ is also a singleton if and only if the dual problem given in~\eqref{eq.D_RLP} has a unique optimum. The result follows.
\end{proof}

\bproposition
	\label{MVF_Pro}
	The restricted value function~\eqref{eq.RVF_ch2} $z$ is a lower semi-continuous, nonincreasing function that is composed of a minimum of a finite number of polyhedral functions.\eproposition
	\begin{proof}
   The \VF{} is the minimum of a finite number of functions, all of which are polyhedral by Proposition~\ref{epiGraphConvex}. The lower semi-continuity property can be established using a proof similar to that presented in~\citep{nemhauser1988scope}. The nonincreasing property of the VF follows from the fact that it is the minimum of a set of nonincreasing functions and must therefore be nonincreasing itself.
\end{proof}

\begin{proof}[Proof of Proposition~\ref{RVFDirDev}]
 Let $d \in \Re^{\ell}$, and a minimal description $\Smin$ of $z$ be given. By Theorem~\ref{thm:Dis_Representation}, $z$ is the minimum of a finite set of functions, $\overline{z}(\cdot;\ x_I)$, for each $x_I\in\Smin$, which is a finite set. Furthermore, by Proposition~\ref{epiGraphConvex}, each of these functions is polyhedral. Thus, Proposition~\ref{prop:dirderivminpolyfunc} (in Appendix~\ref{appendix-ddminfinitepolyhedral}), which is a general result about functions that are the minimum of a finite number of polyhedral functions, applies and yields the first part.

		The second part follows by substituting the formula for the directional derivative of the bounding function from Proposition~\ref{prop:differentiable_VF}.
		
	\end{proof}

	\begin{lemma}\label{lem:upifinactive}
		For any $\zeta\in\C$ and any minimal description $\Smin$, there exists $\epsilon>0$ such that
		\begin{equation}
			\label{eq:upifinactive}
			z(\zeta')\geq \min_{x_I\in S^*_I(\zeta)\cap\Smin} \overline{z}(\zeta';\ x_I), \qquad \forall \zeta'\in B(\zeta;\ \epsilon),
		\end{equation}
		where $B(\zeta;\ \epsilon)$ denotes the open ball of radius $\epsilon$ centered on $\zeta$.
	\end{lemma}
	
\begin{proof}
 Let $\Smin$ be given as per Theorem~\ref{thm:Dis_Representation} and let $\zeta\in\C$. Recall that by Theorem~\ref{thm:Dis_Representation},
		\begin{equation}
			\label{eq:Sminrep}
			z(\zeta')=\min_{x_I\in\Smin} \overline{z}(\zeta';\ x_I), \qquad\forall \zeta'\in\Re^{\ell}.
		\end{equation}
Define 
		$$
		\alpha(\zeta')=\min_{x_I\in S^*_I(\zeta)\cap\Smin} \overline{z}(\zeta';\ x_I), \qquad \forall \zeta'\in\Re^{\ell}.
		$$
Here, the bounding functions $\overline{z}(\cdot;\ x_I)$ for $x_I\in S^*_I(\zeta)\cap\Smin$ are those that are both active at $\zeta$ and included in the given minimal description of the \VF{} provided by $\Smin$. Using the notation just introduced, the inequality~\eqref{eq:upifinactive} simply says $z(\zeta') \geq \alpha(\zeta')$. Now suppose that~\eqref{eq:upifinactive} does {\em not} hold. Since the inequality trivially holds for $\zeta' \not\in \C$ ($z(\zeta')=+\infty$ in that case), then this means that there is some $\zeta'\in B(\zeta;\ \epsilon) \cap \C$ such that for every $\epsilon>0$, there exists $x_I\in\Smin$ with $\overline{z}(\zeta';\ x_I)<\alpha(\zeta')$. This is because $z(\zeta')=\overline{z}(\zeta';\ x_I)$ for some $x_I\in\Smin$, by~\eqref{eq:Sminrep}. Since there are only a finite number of $x_I\in\Smin$, and each bounding function $\overline{z}(\cdot;\, x_I)$ is polyhedral, actually there must exist one $\hat{x}_I\in\Smin$ so that for all $\epsilon>0$, some $\zeta'\in B(\zeta;\ \epsilon)$ has $\overline{z}(\zeta';\ \hat{x}_I)<\alpha(\zeta')$. Now $\overline{z}(\cdot;\ \hat{x}_I)$ is continuous on its finite domain, which is a closed set, so it must be that $\overline{z}(\zeta;\ \hat{x}_I)\leq\alpha(\zeta)$. But $\alpha(\zeta) = z(\zeta)\leq \overline{z}(\zeta;\ \hat{x}_I)$. So it must be that $\overline{z}(\zeta;\ \hat{x}_I)=z(\zeta)$, and $\hat{x}_I\in\S^*_I(\zeta)$. But then $\hat{x}_I\in\S^*_I(\zeta)\cap\Smin$, so in fact $\overline{z}(\zeta';\ \hat{x}_I)\geq\alpha(\zeta')$ for all $\zeta'\in\Re^{\ell}$, and we obtain a contradiction. The result follows.
\end{proof}

  \begin{proof}[Proof of Proposition~\ref{RVFDiff}]
		Let $\Smin$ be a minimal description of the \VF{}, and consider $\zeta\in\C$. Let $g\in \partial_L z(\zeta)$. Then there exists $\epsilon>0$ such that
		$$
		z(\zeta') \ge z(\zeta) + g^{\top}( \zeta' - \zeta), \quad \forall \zeta' \in B(\zeta;\ \epsilon).
		$$
		Let $x_I\in S^*_I(\zeta)$ be chosen arbitrarily. Now for any $\zeta'\in B(\zeta;\ \epsilon)$,
		$$
		\overline{z}(\zeta';\ x_I) \geq
		z(\zeta') \geq z(\zeta) + g^{\top}( \zeta' - \zeta)
		= \overline{z}(\zeta;\ x_I) + g^{\top}( \zeta' - \zeta),
		$$
		where the first inequality follows from the definition of $z$, the second since $g$ is a local subgradient of $z$ at $\zeta$, and the final equality follows since $x_I\in S^*_I(\zeta)$. Thus $g$ is also a local subgradient of $\overline{z}(\cdot;\ x_I)$ at $\zeta$. But $\overline{z}(\cdot;\ x_I)$ is a~\eqref{eq.RLPVF_ch2}, and so is a convex function over a convex finite domain. Hence 
		$$
		g\in \partial \overline{z}(\zeta;\ x_I) = \partial \zLP(\zeta-C^{1:\ell}_Ix_I;\ b-A_Ix_I).
		$$
		Since $x_I$ was chosen arbitrarily from $S^*_I(\zeta)$, we have proved that
		$$
		\partial_L z(\zeta) 
		\subseteq\bigcap_{x_I\in S^*_I(\zeta)}
		\partial\zLP(\zeta-C^{1:\ell}_Ix_I;\ b-A_Ix_I)
		\subseteq\bigcap_{x_I\in S^*_I(\zeta)\cap\Smin}
		\partial\zLP(\zeta-C^{1:\ell}_Ix_I;\ b-A_Ix_I). 
		$$
		To prove containment in the opposite direction, let 
		$$
		g\in \bigcap_{x_I\in S^*_I(\zeta)\cap\Smin}
		\partial\zLP(\zeta-C^{1:\ell}_Ix_I;\ b-A_Ix_I)
		= \bigcap_{x_I\in S^*_I(\zeta)\cap\Smin}
		\partial\overline{z}(\zeta;\ x_I). 
		$$
		By Lemma~\ref{lem:upifinactive}, there must exist $\epsilon>0$ so that \eqref{eq:upifinactive} holds.
		Choose $\zeta'\in B(\zeta;\ \epsilon)$ arbitrarily. Then, by \eqref{eq:upifinactive},
		$$
		z(\zeta')\geq \min_{x_I\in S^*_I(\zeta)\cap\Smin} \overline{z}(\zeta';\ x_I), 
		$$
		and so there exists some $x_I\in S^*_I(\zeta)\cap\Smin$ with
		$$
		z(\zeta')\geq \overline{z}(\zeta';\ x_I) \geq \overline{z}(\zeta;\ x_I) + g^{\top} (\zeta'-\zeta) = z(\zeta) + g^{\top} (\zeta'-\zeta),
		$$
		where the second inequality follows since $g\in\partial\overline{z}(\zeta;\ x_I)$ and the final equality follows since $x_I\in S^*_I(\zeta)$. Thus $g\in \partial_L z(\zeta)$ and it must be that 
		$$
		\partial_L z(\zeta) 
		\supseteq\bigcap_{x_I\in S^*_I(\zeta)\cap\Smin}
		\partial\zLP(\zeta-C^{1:\ell}_Ix_I;\ b-A_Ix_I). 
		$$
		
		Hence
		$$
		\partial_L z(\zeta)
		=\bigcap_{x_I\in S^*_I(\zeta)\cap\Smin}
		\partial \zLP(\zeta-C^{1:\ell}_Ix_I;\ b-A_Ix_I).
		$$
		The result follows by Proposition~\ref{prop:subdiffeqdualopt}.
\end{proof}

\section{Theorem~\ref{thm:Relationship} with $C^{1:\ell} x  = \zeta$}
		\label{appendix-relationship_equality}
		As mentioned in Section~\ref{sec:Relationship}, we can have another representation of Theorem~\ref{thm:Relationship}. Consider the \VF{} $z': \Re^{\ell} \rightarrow \Re \cup \{\pm \infty\}$ as
		\begin{equation}
			z'(\zeta) = \inf_{(x_I,x_C)\in \mathcal{S}(\zeta)} c^{0}_I x_I + c^{0}_C x_C, \tag{RVF$'$} \label{eq.RVF_appendix}
		\end{equation}
		where
		\begin{equation*}
			\mathcal{S} (\zeta) = \left\{(x_I, x_C)\in \Z_+^r \times \Re_+^{n-r}\;:\; C_I^{1:\ell} x_I + C_C^{1:\ell} x_C = \zeta,\ A_I x_I + A_C x_C = b\right\}.
		\end{equation*}
		It is worth noting that the only distinction here is that we have an equality sign for the constraints that serve as our objectives. We formalize the relationship between the \VF{} and the EF through the following theorem. 
		
		\btheorem
		\label{thm:Relationship_appendix} 
		The following statements hold for $\XMO$ and the~\eqref{eq.RVF_appendix} $z'$.
		\begin{enumerate}[label=(\alph*)]
			\item \label{relation-1-appendix} If $(x_I,x_C) \in \XMO$ is an efficient solution (equivalently, $C_I x_I + C_C x_C$ is an NDP), then $(\zeta, c^0x_I+c^0x_C)$ is a point on the boundary of the epigraph of $z'$ for $\zeta = C_I^{1:\ell} x_I + C_C^{1:\ell} x_C$.
			
			\item \label{relation-2-appendix} If $(\zeta, z'(\zeta))$ is a point on the boundary of the epigraph of $z'$ and $\nabla_d z'(\zeta) > 0$ for all $d \in \Re^{\ell}_-\setminus\{\mathbf{0}\}$ for which $\nabla_d z'(\zeta)$ exists, then there exists an efficient solution $(x_I,x_C) \in \XMO$ that yields $z'(\zeta)$ and satisfies $C_I^{1:\ell} x_I + C_C^{1:\ell} x_C = \zeta$.
			
		\end{enumerate} 
		\etheorem

		\begin{proof}
			The proofs for Part~\ref{relation-1-appendix} and Part~\ref{relation-2-appendix} are in accordance with the proofs for Part~\ref{relation-1} and Part~\ref{relation-2} with the exception of the portion designated by $\Leftarrow$ in Theorem~\ref{thm:Relationship}.
		\end{proof}

		Here we have a very similar proposition compared to Proposition~\ref{MVF_Pro} as follows.
		
		\bproposition
		\label{MVF_Pro_appendix}
		The restricted value function (RVF) is a lower semi-continuous function and decreasing over $\C$ where the optimal dual value is negative. Furthermore, it is comprised of a minimum of a finite number of polyhedral functions.
		\eproposition
		\begin{proof}
			The proof for the property of semi-continuity and being a minimum of a finite number of polyhedral functions is the same as that in Proposition~\ref{MVF_Pro}. The decreasing property is straightforward; with negative optimal dual values, the VF is decreasing.\end{proof}

		Therefore, in the steps of the algorithm, although we have equality for the constraints with the parametric RHS, i.e., the constraints that serve as our objectives, we must have negative dual variables $u \in \Re_{-}^{\ell}$, which relates to the constraints with the parametric RHS. In this case, the algorithm generates that part of the VF, which is decreasing and is the same as the related EF. As a result, the VF algorithm remains unchanged in this scenario.

  \section{Proofs of Correctness for the RVF Algorithm}
\label{appendix:correctness}

The following lemmas are used to prove the correctness of the RVF algorithm, given in Theorem~\ref{Correctness}. We start with the following simple observation.

\begin{lemma}\label{lem:CRrearrange}
		For all $(\hat{x}_I,\hat{x}_C)\in\XMO$, we have that
		$$
		\bar{z}(C^{1:\ell}_I\hat{x}_I + C^{1:\ell}_C\hat{x}_C;\ \hat{x}_I)\leq c^0_I\hat{x}_I + c^0_C\hat{x}_C.
		$$
	\end{lemma}
	\begin{proof}
		The bounding function~\eqref{eq.CR_ch2} associated with $\hat{x}_I$ is equivalently expressed as

 \begin{equation}
			\label{eq:CRequiv}
			\bar{z}(\zeta;\ \hat{x}_I) = \min\left\{ c^0_Ix_I+c^0_Cx_C \;:\; (x_I,x_C)\in \XMO,\ x_I = \hat{x}_I,\ C^{1:\ell}_Ix_I + C^{1:\ell}_Cx_C \leq \zeta\right\}.
		\end{equation}
		The result follows as $(\hat{x}_I,\hat{x}_C)$ is feasible for the problem above for $\zeta=C^{1:\ell}_I\hat{x}_I + C^{1:\ell}_C\hat{x}_C$.
	\end{proof}
	
	The following lemma helps to establish that the integer parts added to $\S^k$ by the algorithm eventually suffice to describe the whole VF.
	
	\begin{lemma}\label{lem:thetaismaxoverzeta}
		At iteration $k$ of the algorithm, $\theta^{k+1} = \max_{\zeta \in \C} \bar{z}^k(\zeta)-z(\zeta)$. 
	\end{lemma}
	\begin{proof}
		Let $\zeta^*=\arg\max_{\zeta\in\C} (\bar{z}^k(\zeta)-z(\zeta))$ and denote $\zeta^{k+1} = C^{1:\ell}_Ix^{k+1}_I + C^{1:\ell}_Cx^{k+1}_C$. Since $\zeta^*\in\C$ there must exist $(\hat{x}_I,\hat{x}_C)\in \XMO$ with $\hat{\zeta} := C^{1:\ell}_I\hat{x}_I + C^{1:\ell}_C\hat{x}_C\leq\zeta^*$ and $z(\zeta^*)=c^0_I\hat{x}_I + c^0_C\hat{x}_C = z(\hat{\zeta})$. Further, since $\bar{z}^k$ is nonincreasing, $\bar{z}^k(\hat{\zeta})\geq \bar{z}^k(\zeta^*)$, and hence
		\begin{align*}
			\bar{z}^k(\hat{\zeta}) - z(\hat{\zeta})
			& \geq \bar{z}^k(\zeta^*) - z(\hat{\zeta}) \\
			& = \bar{z}^k(\zeta^*) - z(\zeta^*) \\
			& \geq \bar{z}^k(\zeta^{k+1}) - z(\zeta^{k+1}).
		\end{align*}
		The last of these inequalities follows from the fact that $\zeta^{k+1} \in \C$. Since 
		$$(x_I^{k+1}, x_C^{k+1}) \in 
		\argmax\limits_{(x_I, x_C) \in \XMO} \left(\bar{z}^k(C_I^{1:\ell} x_I + C^{1:\ell}_C x_C) - (c^0_I x_I + c^0_C x_C)\right),
		$$
		it follows that 
		$$
		\bar{z}^k(\zeta^{k+1}) - z(\zeta^{k+1}) \geq
		\bar{z}^k(\hat{\zeta}) - z(\hat{\zeta})
		\Rightarrow
		\bar{z}^k(\zeta^{k+1}) - z(\zeta^{k+1}) =
		\bar{z}^k(\hat{\zeta}) - z(\hat{\zeta}).
		$$
		Finally, by the definition of $\theta^{k+1}$, we have
		$$
		\theta^{k+1} = \bar{z}^k(\zeta^{k+1}) - z(\zeta^{k+1}) = \bar{z}^k(\zeta^*) - z(\zeta^*),
		$$
		and the result follows.
	\end{proof}
	
	The crucial property for establishing termination of the algorithm is that as long as $\bar{z}^k \not= z$, then in iteration $k$, we are guaranteed to produce a point not already contained in $\S^k$. This is established in the lemma below, which states that unless the optimal value of the optimization problem in~\eqref{eq:defnextiterate} is zero, the integer part of the solution obtained is not contained in ${\cal S}^k$. 
	
	\begin{lemma}
		\label{lem:Subprobsolnnewimage}
		Any optimal solution $(x_I^{k+1}, x_C^{k+1})$ to the optimization problem in~\eqref{eq:defnextiterate} having optimal value $\theta^{k+1} > 0$ has the property that $x_I^{k+1} \not\in {\cal S}^k$.
	\end{lemma}
	
	\begin{proof} We show the contrapositive of the lemma as follows. Let $(x_I^{k+1}, x_C^{k+1})$ solve the optimization problem in~\eqref{eq:defnextiterate} having optimal value $\theta^{k+1}$ and suppose that $x_I^{k+1} \in {\cal S}^k$. Then 
		$$
		\bar{z}^k(C^{1:\ell}_Ix^{k+1}_I + C^{1:\ell}_Cx_C^{k+1}) \leq \bar{z}(C^{1:\ell}_Ix^{k+1}_I + C^{1:\ell}_Cx_C^{k+1};\ x^{k+1}_I)\leq c^0_Ix^{k+1}_I+c^0_Cx^{k+1}_C,
		$$
		where the first inequality follows from~\eqref{eq.upper_approx}, since $x^{k+1}_I\in\S^k$, and the second inequality follows from Lemma~\ref{lem:CRrearrange}, since $(x_I^{k+1}, x_C^{k+1})\in\XMO$. It thus follows, by the definition of $\theta^{k+1}$ and $x^{k+1}$, that
		$$
		\theta^{k+1}= \bar{z}^k(C^{1:\ell}_Ix^{k+1}_I + C^{1:\ell}_Cx_C^{k+1})
		- (c^0_Ix^{k+1}_I+c^0_Cx^{k+1}_C)\leq 0,
		$$
		as required to obtain the contrapositive. 
	\end{proof} 
	
		\section{Proof of the Nonempty Stability Region of the Integer Part in Each Iteration}
  \label{appendix-not-wasteful}
		
		In this appendix, we show the details of the proof that the integer part added to $\S^k$ in iteration $k$ has a nonempty stability region \emph{at the time it is added}. Note, however, that this stability region may end up being redundant by the end of the algorithm. First, we show that  the true VF $z(\zeta^*)=c^0_I x_I^{k+1} + c^0_C x_C^{k+1}$ whenever $(x_I^{k+1}, x_C^{k+1})$ solves the optimization problem~\eqref{eq:defnextiterate} and $\zeta^*=C_I^{1:\ell} x_I^{k+1} + C^{1:\ell}_C x_C^{k+1}$. 
		
		\begin{lemma} \label{lem:newitgivesvalue}
			If $(x_I^{k+1}, x_C^{k+1})$ is defined by~\eqref{eq:defnextiterate} then $z(\zeta^*)=c^0_I x_I^{k+1} + c^0_C x_C^{k+1}$ where $\zeta^*=C_I^{1:\ell} x_I^{k+1} + C^{1:\ell}_C x_C^{k+1}$.
		\end{lemma}
 \begin{proof}
			Let $(x_I^{k+1}, x_C^{k+1})$ be defined by~\eqref{eq:defnextiterate} and suppose, for the sake of contradiction, that \linebreak $z(C_I^{1:\ell} x_I^{k+1} + C^{1:\ell}_C x_C^{k+1})\neq c^0_I x_I^{k+1} + c^0_C x_C^{k+1}$. Then there must exist $(\hat{x}_I,\hat{x}_C)$ a solution of the optimization problem that defines the \VF{}, $z(C_I^{1:\ell} x_I^{k+1} + C^{1:\ell}_C x_C^{k+1})$, given by

 \begin{equation*}
				\begin{aligned}
					z(C_I^{1:\ell} x_I^{k+1} + C^{1:\ell}_C x_C^{k+1}) 
					=\min\big\{c^0_Ix_I+c^0_Cx_C\;:\;&(x_I,x_C)\in\XMO,\ \\ C_I^{1:\ell} &x_I + C^{1:\ell}_C x_C\leq C_I^{1:\ell} x_I^{k+1} + C^{1:\ell}_C x_C^{k+1}\big\},
				\end{aligned}
			\end{equation*}
			and it must be that 
			$$
			z(C_I^{1:\ell} x_I^{k+1} + C^{1:\ell}_C x_C^{k+1})= c^0_I\hat{x}_I+c^0_C\hat{x}_C < c^0_I x_I^{k+1} + c^0_C x_C^{k+1}. 
			$$
			Now since $\bar{z}^k$ is nonincreasing, and, by definition, $(\hat{x}_I,\hat{x}_C)$  satisfies 
			$$
			C_I^{1:\ell} \hat{x}_I + C^{1:\ell}_C \hat{x}_C\leq C_I^{1:\ell} x_I^{k+1} + C^{1:\ell}_C x_C^{k+1},$$ 
			it must be that 
			$$
			\bar{z}^k(C_I^{1:\ell} \hat{x}_I + C^{1:\ell}_C \hat{x}_C)\geq \bar{z}^k(C_I^{1:\ell} x_I^{k+1} + C^{1:\ell}_C x_C^{k+1}).
			$$
			Thus 
			$$
			\bar{z}^k(C_I^{1:\ell} \hat{x}_I + C^{1:\ell}_C \hat{x}_C) - (c^0_I\hat{x}_I+c^0_C\hat{x}_C) > \bar{z}^k(C_I^{1:\ell} x_I^{k+1} + C^{1:\ell}_C x_C^{k+1}) - (c^0_I x_I^{k+1} + c^0_C x_C^{k+1}),
			$$
			and $(\hat{x}_I,\hat{x}_C)$ must be a feasible solution for the optimization problem in~\eqref{eq:defnextiterate} having better objective value than $(x_I^{k+1}, x_C^{k+1})$, which is a contradiction.
		\end{proof}
		
		Next, we justify our claim that converting to a (strong) efficient solution via~\eqref{eq:conversion_to_NDP} again yields a solution to~\eqref{eq:defnextiterate}. We omit formal proof that the resulting solution must be efficient for~\eqref{eq.MultiObj_ch2} since this is a straightforward and well-known result from multiobjective optimization.
		
		\begin{lemma} \label{lem:conversionstillopt}
			If $(x_I^{k+1}, x_C^{k+1})$ is defined by~\eqref{eq:defnextiterate} and $(\hat{x}_I,\hat{x}_C)$ solves~\eqref{eq:conversion_to_NDP} then $(\hat{x}_I,\hat{x}_C)$ is also an optimal solution for the optimization problem in~\eqref{eq:defnextiterate}.
		\end{lemma}
		\begin{proof}
			Let $(x_I^{k+1}, x_C^{k+1})$ be defined by~\eqref{eq:defnextiterate} and suppose $(\hat{x}_I,\hat{x}_C)$ solves~\eqref{eq:conversion_to_NDP}. Then $(\hat{x}_I,\hat{x}_C)\in \XMO$, so is feasible for the optimization problem in~\eqref{eq:defnextiterate}. Furthermore,
			$$
			C_I \hat{x}_I + C_C \hat{x}_C\leq C_I x_I^{k+1} + C_C x_C^{k+1},
			$$
			so 
			$$
			c_I^0 \hat{x}_I + c^0_C \hat{x}_C\leq c_I^0 x_I^{k+1} + c^0_C x_C^{k+1},
			$$ 
			and
			$$
			C_I^{1:\ell} \hat{x}_I + C^{1:\ell}_C \hat{x}_C\leq C_I^{1:\ell} x_I^{k+1} + C^{1:\ell}_C x_C^{k+1}.
			$$ 
			Since $\bar{z}^k$ is nonincreasing it must thus be that 
			$$
			\bar{z}^k(C_I^{1:\ell} \hat{x}_I + C^{1:\ell}_C \hat{x}_C)\geq \bar{z}^k(C_I^{1:\ell} x_I^{k+1} + C^{1:\ell}_C x_C^{k+1}),
			$$ 
			and so 
			$$
			\bar{z}^k(C_I^{1:\ell} \hat{x}_I + C^{1:\ell}_C \hat{x}_C) - (c_I^0 \hat{x}_I + c^0_C \hat{x}_C)
			\geq \bar{z}^k(C_I^{1:\ell} x_I^{k+1} + C^{1:\ell}_C x_C^{k+1}) - (c_I^0 x_I^{k+1} + c^0_C x_C^{k+1}).
			$$
			Since $(\hat{x}_I,\hat{x}_C)$ is feasible for the optimization problem in~\eqref{eq:defnextiterate} and has at least as good an objective value as that of $(x_I^{k+1}, x_C^{k+1})$, it must also solve the optimization problem in~\eqref{eq:defnextiterate}. 
		\end{proof}
		
		These two lemmas combine to prove that the integer part added to $\S^k$ at each iteration $k$ of the algorithm has a nonempty stability region. 
		
		\begin{proposition}
			\label{prop:newitnonemptystab}
			If $x_I^{k+1}$ is added to $\S^k$ in iteration $k$ of the~\ref{RVFAlgo_ch2} then its stability region $\C(x_I^{k+1})$ is nonempty.
		\end{proposition}
		\begin{proof}
			By Lemmas~\ref{lem:newitgivesvalue} and~\ref{lem:conversionstillopt}, if $x_I^{k+1}$ is added to $\S^k$ in iteration $k$ of the~\ref{RVFAlgo_ch2}, then $x_C^{k+1}$ found in the process has $c^0_Ix_I^{k+1} + c^0_Cx_C^{k+1}=z(\zeta)$ where $\zeta=C_I^{1:\ell} x_I^{k+1} + C^{1:\ell}_C x_C^{k+1}$. But by the characterization of the VF as the minimum of a finite set of bounding functions (as per Theorem~\ref{thm:Dis_Representation}), we have that 
			$$
			z(\zeta) \leq \bar{z}(\zeta;\ x_I^{k+1})\leq c^0_Ix_I^{k+1} + c^0_Cx_C^{k+1},
			$$ 
			where the latter inequality follows by Lemma~\ref{lem:CRrearrange}, since $(x_I^{k+1},x_C^{k+1})\in\XMO$. Thus it must be that $z(\zeta) = \bar{z}(\zeta;\ x_I^{k+1})$, so $\zeta \in \C(x_I^{k+1})$ and the stability region $\C(x_I^{k+1})$ is nonempty, as required.
		\end{proof}

		\section{Directional Derivatives of the Minimum of Polyhedral Functions}
		\label{appendix-ddminfinitepolyhedral}

		Here, we describe the directional derivatives of a function formed by taking the minimum of a finite set of polyhedral functions.  We are particularly interested in functions defined over the extended reals. 
		
		The properties we identify here are for any such function; they are not specific to VFs. We think it is very likely that the material we provide here has already been published. However we have been unable to locate a source for it, so include it for the sake of completeness.
		
		First, we give definitions and conventions for directional derivatives of functions defined over the extended reals. Then we briefly review the properties of polyhedral functions, before giving our main result on the directional derivatives of the minimum of a finite set of polyhedral functions.
		
		Consider $f:\Re^n\rightarrow \Re\cup\{+\infty\}$, a function defined over the extended reals that is bounded below, and so cannot have the value $-\infty$. We refer to the set of points $x\in\Re^n$ for which $f(x)<+\infty$ as the {\em finite domain} of $f$. We are interested in functions with closed finite domain. The directional derivative of $f$ at the point $x\in\Re^n$ in direction $d\in \Re^n$ is denoted by $\nabla_df(x)$ and is defined by
		$$
		\nabla_d f(x) =  \lim_{t\rightarrow 0^+} \frac{f(x+td)-f(x)}{t}.
		$$
		
		We take $\nabla_df(x)=0$ if $f(x)=+\infty$ (meaning $x$ is not in the finite domain of $f$). If $f(x) < +\infty$ and there exists $\epsilon>0$ such that $f(x+td)=+\infty$ for all $t\in (0,\epsilon)$, then we say $x$ is on the boundary of the finite domain of $f$ and $d$ points out of it. In this case, $\nabla_df(x) = +\infty$.
		
		We now turn our attention to the case of a function defined as the minimum of a finite set of polyhedral functions.
		
		Say $z^k:\Re^n\rightarrow\Re\cup\{+\infty\}$ is a given polyhedral function for each $k=1,\dots,K$, which means that the epigraph over its finite domain is a polyhedron. A property of a polyhedral function is that for any point $x$ and direction $d$, one of the following cases must hold:
		\begin{enumerate}
			\item $x$ is not in the finite domain of $z^k$, in which case $\nabla_d z^k(x)=0$,
			\item $x$ is on the boundary of the finite domain of $z^k$ and $d$ points out of it, in which case $\nabla_d z^k(x)=+\infty$, or
			\item there exists $\epsilon>0$ such that $x+td$ is in the finite domain of $z^k$ for all $t\in[0,\epsilon]$, and, furthermore, $\epsilon$ may be taken to be sufficiently small that all points in the line segment defined as $\{x+td\;:\; 0\leq t \leq \epsilon\}$ lie on the same facet of the polyhedron created by the epigraph of $z^k$ over its finite domain. In this case, there exists $a\in\Re^n$ and $b\in\Re$ such that $z^k(x+td)=a^{\top}(x+td)+b$ for all $t\in[0,\epsilon]$. Note that $a$ and $b$ depend on $x$ and $d$, but not on $\epsilon$, and it follows that $\nabla_d z^k(x)=a^{\top}d$. 
		\end{enumerate}
		It is also helpful to observe that the finite domain of a polyhedral function must itself be a polyhedron, which is closed.

		Let  $z:\Re^n\rightarrow\Re\cup\{+\infty\}$ be defined to be
		$$
		z(x) = \min_{k=1.\dots,K} z^k(x).
		$$
		Since each function $z^k$ bounds $z$ from above, we call $z^k$ a bounding function. Define
		$$
		\kappa^*(x) = \{k\in\{1,\dots,K\}: z^k(x) = z(x)\},
		$$
		to be the index set of the bounding functions that are active at $x$. Clearly, the finite domain of $z$ is the union of the finite domains of its bounding functions, so is the union of a finite collection of closed sets, and hence is closed. 
		
		It is useful to observe that in the case that $x$ is in the finite domain of $z$ and $d$ does {\em not} point out of it, then it may be that there is a jump, or discontinuity, in $z$ when moving away from $x$ in the direction $d$. In other words, for such $x$ and $d$, it may be that 
		$$
		\lim_{t\rightarrow 0^+} z(x+td) > z(x).
		$$
		In such a case, there must exist $k\in \{1,\dots,K\}$ and $\epsilon > 0$ so that $z(x+td)=z^k(x+td)$ for all $t\in (0,\epsilon)$, where, importantly, $t=0$ is {\em not} included in this interval. So
		$$
		\lim_{t\rightarrow 0^+} z(x+td) 
		=\lim_{t\rightarrow 0^+} z^k(x+td) = z^k(x) > z(x),
		$$
		and it follows that $\nabla_d z(x) = +\infty$.
		
		Whether or not $z$ has a discontinuity when moving away from $x$ in direction $d$, the structure of $z$ as a minimum of a finite set of polyhedral functions ensures that when $x$ is in the finite domain of $z$ and $d$ does not point out of it, there must exist $k\in \{1,\dots,K\}$, $a\in\Re^n$ and $b\in\Re$ defining a facet of the epigraph of $z^k$ over its finite domain, and $\epsilon>0$ sufficiently small so that 
		$$ 
		z(x+td) = z^k(x+td) = a^{\top}(x+td) + b, 
		\qquad \forall t\in (0,\epsilon).
		$$
		We say that the facet of $z^k$ defined by $(a,b)$ {\em yields $z$ when moving away from $x$ in direction $d$}.
		
		We now give the main result.
		
		\begin{proposition} \label{prop:dirderivminpolyfunc}
			For any point $x\in\Re^n$ and any direction $d\in\Re^n$,
			$$
			\nabla_dz(x) = \min_{k\in\kappa^*(x)} \nabla_d z^k(x).
			$$
		\end{proposition}
		
		\begin{proof} Let the point $x\in\Re^n$ and direction $d\in\Re^n$ be given. There are three main cases.\\
			\underline{Case 1: $x$ is not in the finite domain of $z$.} \\
			In this case, $\nabla_d z(x)= 0$. Furthermore, since $z(x)=+\infty$, it follows from the definition of $z$ that $z^k(x)=+\infty$, and hence $\nabla_d z^k(x)= 0$, for all $k=1,\dots,K$. The result follows.\\
			\underline{Case 2: $x$ is in the boundary of the finite domain of $z$ and $d$ points out of it.} \\
			This means that there exists $\epsilon>0$ so that $z(x+td)=+\infty$ for all $t\in (0,\epsilon)$. By the definition of $z$, it must be that for all $k=1,\dots,K$, $z^k(x+td)=+\infty$ for all $t\in (0,\epsilon)$. Now for all $k\in\kappa^*(x)$, $x$ is in the finite domain of the bounding function $z^k$, since $z^k(x)=z(x) < +\infty$. Thus $\nabla_d z^k(x)=+\infty$ and the result follows.\\
			\underline{Case 3. there exists $\epsilon>0$ such that $z(x+td)<+\infty$ for all $t\in [0,\epsilon)$.}\\
			Note that $t=0$ is included in the definition of this case, since $x$ not in the finite domain of $z$ was covered in Case 1. Assume, without loss of generality, that $k=1$ is the index of the bounding function having a facet that yields $z$ when moving away from $x$ in direction $d$. Thus there must exist $\epsilon'>0$, $a\in\Re^n$ and $b\in\Re$ such that $$
			z(x+td) = z^1(x+td)=a^{\top}(x+td)+b, \qquad \forall t\in (0,\epsilon').
			$$
			Let $k^*$ denote the minimizer of $\nabla_d z^k(x)$ over $k\in \kappa^*(x)$. Note that $\nabla_d z^1(x) = a^{\top}d$, which is finite. So $1\in\kappa^*(x)$ implies $\nabla_d z^{k^*}(x)$ is finite. Thus if $\nabla_d z^{k^*}(x)=+\infty$ it must be that $1\not\in\kappa^*(x)$. Hence 
			$$
			\lim_{t\rightarrow 0^+} z(x+td) = \lim_{t\rightarrow 0^+} z^1(x+td) = z^1(x) > z(x),
			$$
			and $\nabla_d z(x)=+\infty=\nabla_d z^{k^*}(x)$, as required. Now consider the case that $\nabla_d z^{k^*}(x)$ is finite. By the properties of polyhedral functions discussed above, there must exist $\epsilon''>0$, $\hat{a}\in\Re^n$ and $\hat{b}\in\Re$ such that $z^{k^*}(x+td)=\hat{a}^{\top}(x+td)+\hat{b}$ for all $t\in [0,\epsilon'')$, so $\nabla_d z^{k^*}(x)=\hat{a}^{\top}d$. To summarize, we have
			$$
			\min_{k\in\kappa^*(x)} \nabla_d z^k(x)=\hat{a}^{\top}d \qquad\text{and}\qquad \nabla_d z(x)=a^{\top}d.
			$$
			Now suppose, for the sake of contradiction, that $\hat{a}^{\top}d \neq a^{\top}d$. Define $\epsilon'''$ by
			$$
			\epsilon''' = \left\{\begin{array}{ll}
				+\infty, & \text{if } \hat{a}^{\top}d < a^{\top}d \\
				\frac{z^1(x)-z(x)}{\hat{a}^{\top}d - a^{\top}d}, & \text{if } \hat{a}^{\top}d > a^{\top}d.
			\end{array}\right.
			$$
			Observe that $\epsilon'''>0$ since $1\in\kappa^*(x)$ implies $\hat{a}^{\top}d =\nabla_d z^{k^*}(x)\leq \nabla_d z^1(x) = a^{\top}d$, so the case $\hat{a}^{\top}d > a^{\top}d$ implies that $1\not\in\kappa^*(x)$, and so $z^1(x) > z(x)$.
			Then for $0<t<\epsilon'''$ we have that
			$$
			t(\hat{a}^{\top}d - a^{\top}d) < z^1(x)-z(x).
			$$
			Now $z(x)=\hat{a}^{\top} x + \hat{b}$ since $k^*\in\kappa^*(x)$, and by continuity of $z^1$, we also have $z^1(x)=a^{\top}x+b$. Substituting these in, we obtain
			$$
			t(\hat{a}^{\top}d - a^{\top}d) < a^{\top}x+b-(\hat{a}^{\top} x + \hat{b}),
			$$
			which is equivalently written as
			$$
			\hat{a}^{\top} (x+td) + \hat{b} < a^{\top}(x+td)+b.
			$$
			Since $t>0$, $t <\epsilon'$ and $t<\epsilon''$, it must be that
			$$
			z^{k^*}(x+td) = \hat{a}^{\top} (x+td) + \hat{b} < a^{\top}(x+td)+b = z(x+td),
			$$
			which contradicts the definition of $z$. Thus it must be that $\hat{a}^{\top}d = a^{\top}d$, as required.
		\end{proof}

	\end{appendices}
	
\end{document}